\documentclass[12pt,a4paper,final]{article}
\usepackage[margin=25mm]{geometry}
\usepackage[utf8]{inputenc}

\usepackage{hyperref}

\usepackage{aligned-overset,amsmath,amsfonts,amssymb,%
        bbm,comment,datetime,bm,relsize,tikz}
\usepackage[normalem]{ulem}
\usepackage[notref,notcite,color]{showkeys}

\newcommand{\Bh}{\mathsf{Bh}}
\newcommand{\He}{\mathsf{He}}
\newcommand{\glo}{\mathrm{glo}}
\newcommand{\loc}{\mathrm{loc}}
\newcommand{\GE}{\mathsf{GE}}
\newcommand{\bbone}{\mathbbm{1}}

\newcommand{\ba}{\begin{array}} \newcommand{\ea}{\end{array}}
  
\newcommand{\ds}{\displaystyle} 
\newcommand{\wt}{\widetilde} \newcommand{\wh}{\widehat}
\newcommand{\ol}{\overline} \newcommand{\ul}{\underline}
\newcommand{\eps}{\varepsilon}
\def\dd{\:\mathrm{d}} % differential for integration 
\def\ee{\mathrm{e}}     % Euler's number 
\def\R{{\mathbb R}} \def\N{{\mathbb N}} 
\newcommand{\mafo}{\mathrm}
\newcommand{\dom}{\mathop{\mafo{dom}}}
\newcommand{\pl}{\partial}
\newcommand{\ti}{{\times}}
\newcommand{\tdfrac}[2]{\tfrac{\ds#1}{\ds#2}}

\newcommand{\bigset}[2]{ \big\{\, #1 \: \big| \: #2 \,\big\} }
\newcommand{\Bigset}[2]{ \Big\{\: #1 \; \Big| \; #2 \:\Big\} }
\newcommand{\norm}[2]{\pmb{\bm|} #1 \pmb{\bm|}_{#2} }

\DeclareMathOperator{\id}{id}
\DeclareMathOperator{\supp}{supp}
\newtheorem{theorem}{Theorem}[section]
\newtheorem{lemma}[theorem]{Lemma}
\newtheorem{definition}[theorem]{Definition}
\newtheorem{proposition}[theorem]{Proposition}
\newtheorem{corollary}[theorem]{Corollary}
\newtheorem{remark}[theorem]{Remark}
\newtheorem{example}[theorem]{Example}

\newcommand{\ProofFont}{\bfseries}
\newcommand{\ProofFinish}{. }
\newcommand{\PROOF}{Proof}%{{\ProofFont Proof\ProofFinish{} } }
\newcommand{\QED}{\mbox{}\hfill\rule{5pt}{5pt}\medskip\par}
\newenvironment{proof}[1][\PROOF]{{\ProofFont#1\ProofFinish}}{\QED}

\def\bbG{{\mathbb G}} \def\bbH{{\mathbb H}} 
 \def\bbK{{\mathbb K}} 
 \def\bbN{{\mathbb N}} 
\def\bbP{{\mathbb P}}  \def\bbR{{\mathbb R}}

\def\calA{{\mathcal A}}  
\def\calD{{\mathcal D}} \def\calE{{\mathcal E}}

\def\calP{{\mathcal P}}  \def\calR{{\mathcal R}}

  \def\rmc{{\mathrm c}} 
\def\rmd{{\mathrm d}}

\def\rmA{{\mathrm A}}  \def\rmC{{\mathrm C}} 
\def\rmD{{\mathrm D}}   
 \def\rmH{{\mathrm H}}  
  \def\rmL{{\mathrm L}}

 \def\rmW{{\mathrm W}}  
 
\def\FG{\mathbf}

  \def\bfL{{\FG L}} 
   
\def\bfP{{\FG P}}   
 \def\bfT{{\FG T}}

\numberwithin{equation}{section}
\numberwithin{figure}{section}

\begin{document}

%\title{Quasistatic one-dimensional viscoelasticity\\ as gradient flow}
\title{Gradient-flow characterizations of \\ 
 one-dimensional quasistatic viscoelasticity\\ 
with Bhattacharya-like viscosity}

\author{ Alexander Mielke%
\thanks{Weierstraß-Institut f\"ur Angewandte
   Analysis und Stochastik, Wilhelm-Anton-Amo-Str.\,39, 10117 Berlin,
   Germany, \texttt{alexander.mielke@wias-berlin.de}.}
\ 
and Billy Sumners%
 \thanks{Maxwell Institute for Mathematical Sciences, 
    Department of Mathematics, Heriot-Watt University, Edinburgh, 
    UK, EH14 4AS, \texttt{w.b.sumners@sms.ed.ac.uk}. 
\hfill
\tiny \usdate\today,\currenttime}
}

\date{22 April 2026}
 
\maketitle

\vspace{-1em}

{\itshape \small 
 \begin{center}
    It is our pleasure to dedicate this work to Robin Knops, in admiration of his foundational contributions to the analysis and community of linear and nonlinear elasticity
 \end{center}
}

% {\footnotesize \tableofcontents} 

\vspace{0.5em}

\begin{abstract}
 We study the equation of one-dimensional quasistatic nonlinear
viscoelasticity with Dirichlet boundary conditions, in the particular case that
the underlying dissipation geometry (provided by the viscosity) is comparable
to the Bhattacharya metric on probability densities. We establish a global
existence result for weak solutions, with an approach based on a spatial
discretization allowing us to work directly with the Riemannian metric
associated to the viscosity. Strong convergence of spatially discrete solutions
is shown directly -- this is possible thanks to Lipschitz estimates achieved
locally on energy sublevels enabled by an explicit derivation of the stretching
of tangent vectors under the flow in the discrete setting and the relationship
to the Bhattacharya metric. We furthermore prove gradient-flow representations
for the solutions: they are curves of maximal slope and, under a
global convexity hypothesis on the energy sublevels, we prove they satisfy a
metric evolutionary variational inequality.
% 53C22  % Geodesics
\end{abstract}

\section{Introduction}

We consider elastoplastic deformations of a one-dimensional bar occupying the
reference domain $\Omega=(0,\ell)\subset \R^1$.  The deformation of the bar is
restricted by the Dirichlet boundary conditions $u(t,0)=0$ and
$u(t,\ell)=\mu>0$. Without loss of generality we will assume $\mu = \ell = 1$
from now on. Hence, the strain variable $p(t)=p(t,x)=\pl_x u(t,x) >0 $ is
assumed to satisfy the condition
\[
 p(t,\cdot) = 
  \bfP := \biggl\{
    p\in \rmL^1(\Omega) \biggm| \int_\Omega p \dd x=1, \ 
    p > 0 \text{ a.e.} 
  \biggr\}.
\]
Hence, $ \bfP $ stands for positive probability densities in $\rmL^1(\Omega)$.  

The equation for one-dimensional viscoelasticity (see
e.g.\ \cite{Dafe69MIBV, KutHic88IBVP, AntSei96QHPE, Seng21NVSR} and
\cite[XVI.3]{Antm95NPE}) reads 
\begin{equation}
  \label{eq:I.ViscoElast}
  \rho \ddot u = \pl_x \Big( W'(\pl_x u) -g(x)  + \nu(\pl_x u) \pl_x\dot u\Big).  
\end{equation}
Assuming quasistationarity means neglecting inertia, see \cite{NaOrRe94BHHE,
  DonMar04CSLM, MiPeMa08CSKV, KnoQui18QSAL, Knop19QSAI, KroRou20QVSC} for the
justification of quasistationarity in similar cases. This limit amounts to
setting $\rho=0$, and we can then integrate the equation once in
$x \in \Omega$. Upon using $p\in \bfP$ and defining $k(p)=k_p=1/\nu(p)$, we
find the equation
\begin{equation}
  \label{eq:I.BasicEq} 
  \dot p = -\bbK(p)\big( W'(p)- G(x)\big) \quad\text{with}\quad 
  \bbK(p) \xi  := k_p \biggl( \xi  - \frac{[k_p \xi ]}{[k_p]} \biggr) 
\end{equation}
(see also \cite{MiOrSe14ANVM, BalSen15QNVG, ParPeg25CNNG})
where we use the averaging notation
\[
  [f] := \int_\Omega f(x) \, \rmd{x}, \quad f \in \rmL^1(\Omega).
\]
The previous works on this model often use constant (inverse) viscosity
$ k\equiv k_0$ corresponding to the case of a Kelvin solid, see e.g.\
\cite{Pego92SGSC, BalSen15QNVG, ParPeg25CNNG}. In \cite{MiOrSe14ANVM} the more
general class $ k(p)=p^\alpha$ is considered, but the main existence result is
restricted to the case $k(p)=\kappa p$. In this paper we allow general inverse
viscosities satisfying
\begin{equation}
  \label{eq:k(p)bounds}
  \exists\,  \ol\kappa\geq\ul\kappa >0\ \forall\, p\geq 0:\quad  
\ul\kappa\, p\leq 
k(p) \leq \ol\kappa\, p.
\end{equation}
However we expect that many of our results can be transferred to more general
situations.  Hence, for $p\in \bfP$ we have
$ \ul\kappa \leq [k_p]\leq \ol\kappa $ which will simplify many of our
arguments but excludes the Hilbert-space gradient-flow theory corresponding to
$\nu=k\equiv 1$ as treated in \cite{Pego92SGSC, BalSen15QNVG, ParPeg25CNNG}.

Our approach is strongly motivated by the general theory of gradient systems as
described in \cite{AmGiSa05GFMS, Miel23IAGS} and is based on the gradient
system $ \big( \bfP, \calE, \bbK \big)$, along with finite-dimensional
approximations $ \big( \bfP_N, \calE_N, \bbK_N \big)$ inspired in part by
\cite{BalSen15QNVG}.  The total elastic stored energy is given by the energy
functional
\[
\calE(p)=\int_\Omega\Big(  W(p(x)) - G(x) p(x)\Big) \dd x ,
\]
where $W:{]0,\infty[} \to \R$ is the stored energy density which is continuous and
satisfies $W(p)\to + \infty$ for $p\searrow 0$ and for $p\to +\infty$. Further
properties of $W$ will be imposed later.  Usually the loading term in the
energy is written $-\int_\Omega g(x) u(x)\dd x$ but defining $G(x)=-\int_x^1
g(x')\dd x' $ and using $u(0)=0$ we can replace this by $\int_\Omega G p \, \rmd{x}$. 
Since $g\in \rmL^1(\Omega)$ is a standard assumption we have $G\in
\rmW^{1,1}(\Omega) \subset \rmC^0(\ol\Omega)$. 

The viscosity is encoded via  the linear Onsager operator $ \xi \mapsto
\bbK(p) \xi $ that depends on the state $p \in \bfP$ and can be seen as the
self-adjoint operator associated with the 
nonnegative quadratic form  
\begin{equation}
  \label{eq:def.calR*}
  \calR^*(p, \xi )
  := \frac{1}{2} \langle \xi, \bbK(p)\xi \rangle = 
   \frac12 \int_\Omega k_p \biggl( \xi- \frac{[k_p \xi ]}{[k_p]} \biggr)^2 \dd x
  = \frac12 \int_\Omega k_p \xi^2 \dd x - \frac{[k_p \xi ]^2}{[k_p]} .
\end{equation}
The associated gradient-flow equations are given abstractly in the form
\[
\dot p = - \bbK(p) \,\rmD\calE(p) \qquad \text{or} \qquad 
\dot p^N = - \bbK_N(p^N) \,\rmD\calE_N(p^N),
\]
such that the first equation equals \eqref{eq:I.BasicEq}, and the second is the
finite-dimensional approximation.  

The functional $\calR^*$ is called the dual dissipation potential arising from
the viscosity law
$\bfT_\mafo{visc}(\pl_x{u},\pl_x \dot u) = \nu(\pl_x u) \pl_x \dot u$. The primal
dissipation potential is defined via
\begin{equation}
  \label{eq:def.calR}
  \calR(p, \dot p)
  := \frac{1}{2} \langle \bbG(p) \dot p, \dot p \rangle
  = \frac12 \int_\Omega \nu(p) \,(\dot p)^2 \dd s 
  =  \frac12 \int_\Omega \frac{(\dot p)^2}{k(p)}  \dd x
\end{equation}
such that $\bfT_\mafo{visc}(p,\dot p)=\pl_{\dot p} \calR(p,\dot p)$.  The main
idea of this paper is to consider the viscous dissipation potential as a
state-dependent Riemann tensor $\bbG$ on the state space $\bfP$ and do as much
as possible of the analysis for equation \eqref{eq:I.BasicEq} using this
infinitesimal metric and an associated global distance
$\calD^k:\bfP\ti \bfP \to {[0,\infty[}$, which is not defined a priori to be
the usual geodesic distance, but is a suitable limit of geodesic distances
$\calD^k_N$ on the finite-dimensional approximation space $\bfP_N$. This allows
us to quite readily transfer results from the discrete situation.

We mention that the case $\nu(p)=\nu_0/p $ or equivalently $k(p)=\kappa\, p$
plays a special role. One the one hand, it arises as the Lagrangian form of the
linear Newtonian viscosity of a fluid (see e.g.\
\cite[Eqn.\,(4.29)]{ZaPeTh23GFRF}) and hence has a clear mechanical
interpretation. On the other hand, the distance $\calD^k$ can be obtained
explicitly in terms of the so-called Bhattacharya distance $\Bh$ or spherical
Hellinger distance, see e.g.\ \cite{LasMie19GPCA,Miel25?NHDV}. 
More precisely $k(p)= \kappa\, p $ gives the explicit relations 
\begin{equation*}
  \calD^k(p_0,p_1)
  = \frac{2}{\sqrt{\kappa}} \, \Bh(p_0,p_1)
  = \frac{4}{\sqrt{\kappa}} \arcsin\biggl( \frac12 \,\He(p_0,p_1)\biggr) \text{ with } 
  \He(p_0,p_1) = \bigl\lVert \sqrt{p_1} - \sqrt{p_0} \bigr\rVert_{\rmL^2} .
\end{equation*}
In particular, assumption \eqref{eq:k(p)bounds} implies $\frac2{\sqrt{\ol \kappa}}
\Bh(p_0,p_1) \leq \calD^k(p_0,p_1) \leq \frac2{\sqrt{\ul\kappa}} \Bh(p_0,p_1)$,
which explains why we call $ p \mapsto \nu(p)=1/k(p)$ a 
\emph{Bhattacharya-like viscosity}. 

The advantage of using $\calD^k$ rather than the more explicit $\Bh$ will be
that working with the exact inverse viscosity law identifies the intrinsic
dissipation structure, and this will allows us to derive very precise results
on the stability of solutions, namely so-called contraction/expansion
estimates.  Working with $\Bh$ will lead only to the less precise Lipschitz
estimates. For the former, it will be essential to use a fine interplay between
the viscosity law and the elastic storage properties. This interplay manifests
itself in the following three important conditions (see \eqref{eq:k.W.conds}
and \eqref{eq:Monotone})
\begin{equation}
  \label{eq:I.estim}
  \begin{gathered}
    \exists \, \lambda_W \in \R\ \forall \, p>0: \quad k(p) W''(p) + \frac{k'(p)}{2}
    W'(p) \geq \lambda_W, \\
    \exists\, B_1,B_2,B_3\in \R\ \forall p>0: \quad 
    \bigl\lvert k(p)W'(p) \bigr\rvert \leq B_1 W(p) + B_2+B_3(p{-}1),
   \\
    \exists \, p_* > 1\ \forall \, p \in {\bigl] 0,\tfrac1{p_*} \bigr]} \cup 
      {[p_*, \infty[}:  \quad k(p)\, W''(p) + \frac{k'(p)}2 W'(p) \geq 0.
  \end{gathered}
\end{equation}
The second condition clearly relates to the stress control conditions 
used in \cite{Ball84MELE,Ball02SOPE,KnZaMi10CGPM,MiRoSa18GERV}. The other
two conditions can be seen as a convexity requirement for $W$
in terms of the viscosity-adapted strain measure 
\[
b(p):= \int_0^p \!\!\frac{1}{\sqrt{k(q)}}\,\dd q .
\]
If we write $W(p)= \wt W(b(p))$, then the first and third conditions may be
rewritten as simple semi-convexity conditions, viz.  
\[
\pl_b^2 \wt W(b) \geq \lambda_W \ \text{ for all } b>0\qquad \text{and} \qquad
\pl^2_b\wt W(b) \geq 0 \ \text{ for }b\leq \frac1{b_*} \text{ or }b\geq b_*.
\] 

The plan of the paper is as follows. In Section \ref{se:EvolStep} we state all
the assumptions on our model explicitly and discuss discretized solutions
$t \mapsto p^N(t)\in \bfP_N:=\bbP_N\bfP$ that are obtained as step functions on
a uniform grid with $N$ intervals of length $1/N$. Here $\bbP_N$ is the $\rmL^2$
orthogonal projection onto these step functions.  For this we replace
$G(x)=-\int_x^1 g(x') \dd x'$ by a piecewise constant interpolant
$ G^N=\bbP_N G$ and show global existence for the ODE of dimension
$(N{-}1)$. Of course, the energy $t \mapsto \calE_N(p^N(t))$ along solutions is
nonincreasing. The main result in Theorem \ref{th:Stretching} concerns an
explicit formula for the stretching of tangent vectors $y^N$ along the
solutions $p^N$ with $\calE_N(p^N(t))\leq \calE_N(p^N(0))\leq E$ in the form
\[
\frac{\rmd}{\rmd t}\, \calR\big(p^N(t),y^N(t)\big) 
\leq - \Lambda^\mafo{inf}_N(E) \: \calR\big(p^N(t),y^N(t)\big), 
\]
where the essential point is that the infinitesimal sublevel stretching rate
$- \Lambda^\mafo{inf}_N(E) $ can be bounded from above by a function
$-\bfL^\mathrm{inf}(E)$ depending only on the energy level $E$ and not on
the discretization number $N$, see Corollary \ref{co:LowBddLambda}. For this,
the first two estimates in \eqref{eq:I.estim} are crucial.

In Section \ref{se:GeodDist} we discuss the discrete geodesic distances
$\calD^k_N$ on $\bfP_N$, their limit $\calD^k$ on $\bfP$, and their relation to
the Bhattacharya distance $\Bh$. For the latter, there are explicit expressions
for the distance and for the geodesic curves. This will be used in Section
\ref{se:LocContrGlobLip} to derive a local contraction/expansion estimate for
discrete solutions $t \mapsto p^N_j(t)$, $j=0,1$, lying in a sublevel
$\Sigma_N(E):=\bigset{p\in \bfP_N}{ \calE_N(p^N)\leq E}$:
\begin{align*}
& \forall \, E>\inf \calE\ \ \exists\, \Delta_E>0 \quad \text{such that}
\\
& \forall \, N \in \N \ \forall \, p^N_0(0),p_1^N(0) \in \Sigma_N(E) 
\text{ with } \calD^k(p_0^N(0),p_1^N(0))\leq \Delta_E:
\\[0.3em]
&\qquad\quad \calD^k_N\big( p^N_0(t), p^N_1(t) \big) 
    \leq \ee^{- \bfL^\mathrm{inf} (E{+}2)t} \,
\calD^k_N\big( p^N_0(0), p^N_1(0) \big) \quad \text{for all }t\geq 0. 
\end{align*}
We emphasize that the name ``contraction/expansion estimate'' relates to the
fact that there is a factor ``$1$'' in front of the exponential term, otherwise
it will be called a Lipschitz estimate. 

Again the stretching rate $-\bfL^\mathrm{inf}(E{+}2)$ is independent of $N$ and
with a similar estimate we can control the distance between solutions $p^N(t)$
and $p^M(t)$ and prove strong convergence $p^N(t)\to p(t)$, thus defining a
global semigroup $\big( S_t\big)_{t>0}$ on
$\dom(\calE):=\bigset{ p \in \bfP}{\calE(p)< \infty}$, see Theorem
\ref{th:SGLocContract}. We again have a local contraction/expansion estimate on
sublevels as above and, under a natural additional condition on $W$, we have a
global Lipschitz estimate such that for all energy levels $E> \inf \calE$ there
exists $\bfL^\mafo{glob}(E)>-\infty$ such that for
$p_0,p_1 \in \Sigma(E):=\bigset{ p \in \bfP}{ \calE(p)\leq E}$ we have
\[
 \Bh\big( S_t(p_0), S_t(p_1)  \big) \leq \sqrt{\ol\kappa /\ul\kappa }\:
\ee^{-\bfL^\mafo{glob}(E)t} \: \Bh\big( p_0, p_1 \big) \quad \text{for all }t\geq 0. 
\]

The final Section \ref{se:CharSemiflows} shows that the limit curves
$t \mapsto p(t)=S_t(p(0))$ are indeed weak solutions of the original equation
\ref{eq:I.BasicEq}. Moreover, Theorem \ref{th:CurvesMaxSlope} shows that the
solutions are \emph{curves of maximal slope} for the gradient system
$ \big( \bfP, \calE, \calR \big)$ in the sense of \cite{AmGiSa05GFMS}. Here we
do not need the distance $\calD^k$ but use the primal and dual dissipation
potentials $\calR$ and $\calR^*$, for which we establish the validity of a
chain rule and obtain the absolute continuity of $t \mapsto \calE(p(t))$ as
well as
\[
  \frac\rmd{\rmd t} \calE(p(t)) = -\calR\big(p(t),\dot p(t)\big) 
  - \calR^*\big(p(t), -\rmD\calE(p(t) \big) \quad \text{for a.a. } t>0.
\]
Moreover, we conclude the energy convergence $ \calE(p^N(t)) \to \calE(p(t))$
for $ N \to \infty $.

In Section \ref{su:SublevelEVI} we show that the solutions also satisfy the
so-called sublevel Evolutionary Variational Inequality (EVI). Based on the
results in \cite{MurSav20GFEV} we first do this for the discretized solutions
$p^N$ and then pass to the limit in a derivative-free version of the EVI by
using that the strong convergence $p^N(t)\to p(t)$ which implies
$\calD^k_N(p^N(t),\bbP_N q) \to \calD^k ( p(t), q)$ and the energy
convergence. However, the result is obtained under the assumption that the
global stretching rate $-\Lambda^\mafo{glob}_N(E)$ is uniformly bounded from
above by some $-\Lambda^\mafo{glob}(E)<\infty$ for all $ E>\inf \calE$. So far,
we are only able to show this for the Bhattacharya case $k(p)= \kappa p$, which
leads to the our last result stated in Corollary \ref{co:Bhatt.EVI}.

The results in this paper can be seen as a major extension of
\cite{MiOrSe14ANVM} in three directions: First, we show existence of
weak solutions and curves of maximal slope for a much larger class of models
because we allow quite general functions $p\mapsto k(p)$ instead of the
Bhattacharya case only. Secondly, we obtain local contraction/expansion
estimates in the intrinsic distance as well as global Lipschitz
estimates in the Bhattacharya distance. 
Thirdly, we derive a true EVI formulation on sublevels.  Our analysis is based
on spatial discretization whereas \cite{MiOrSe14ANVM} uses the minimizing
movement scheme, but both approaches have in common that strong convergence is
shown in a direct manner. This is important because the viscoelastic problem
under investigation does not have any compactness properties.

In the parallel work \cite[Chap.\,4]{Sum2026Thesis}, the case
$k(p) \leq k_0$ is treated based on a similar analysis using spatial
discretization. There, local stretching estimates in the vein of
Theorem~\ref{th:Stretching} are derived to demonstrate an EVI
representation. 

For more qualitative properties of the induced gradient flow, e.g.\
convergence of solutions into steady states, we refer to \cite{Pego92SGSC,
  BalSen15QNVG, Seng21NVSR, ParPeg25CNNG}.

\section{Evolution for step functions}
\label{se:EvolStep}

We consider piecewise constant approximations, which are in fact also exact
solutions if the function $G$ is replaced by its piecewise constant
interpolant. For $ N \in \N $ we define the step functions using the 
partition points $\tau^N_i:= i /N$ as follows:
\begin{equation}
  \label{eq:ApproxN}
  p^N(x) = \sum_{i=1}^N p^N_i \bbone_{{[\tau_{i-1}^N,\tau_i^N[}}(x) . 
\end{equation}
Moreover, we define the $(N{-}1)$-dimensional affine subsets $\bfP_N
\subset \bfP$ and the projection $\bbP_N:\bfP \to \bfP_N$ via 
\[
  \bfP_N:=\bigset{ p^N\in \bfP}{ p^N \text{ as in \eqref{eq:ApproxN}} } \quad
  \text{ and } \quad \bbP_N{\wt{p}} := p^N := \sum_{i=1}^N \wt{p}^N_i
  \bbone_{[\tau_{i-1}^N, \tau_i^N[}.
\]
with $\wt p^N_i = N \int_{\tau_{i-1}^N}^{\tau_i^N} \wt p(y)\dd y $. With this
we define the $(N{-}1)$-dimensional gradient system $(\bfP_N,\calE_N, \bbK_N )$
via
\begin{align*}
 &\calE_N:=\calE|_{\bfP_N} , \text{ namely }\calE_N(p^N) = \int_\Omega \big( W(p^N(x)) - G(x)p^N(x)    \big) \dd x,\\
& \bbK_N(p^N)\xi^N := \bbK(p^N)\xi^N \quad\text{for}\quad \xi^N= \sum_{i=1}^N
      \xi^N_i  \bbone_{{[(i-1)/N,i/N[}}.
\end{align*}
Note that $\calE_N$ does not explicitly depend on $N$, and so the energy
sublevels satisfy 
\begin{equation*}
  \begin{gathered}
    \Sigma(E) := \bigset{p\in \bfP}{ \calE(p)\leq E}, \\
    \Sigma_N(E) := \bigset{p^N\in \bfP_N}{ \calE_N(p^N)\leq E} 
    \ = \ \bfP_N \cap \Sigma(E).
  \end{gathered}
\end{equation*}
However, the derivative $\rmD\calE_N$ is the piecewise constant function
$\rmD\calE_N(p^N)= W'(p^N){-} G^N$ with $G^N:=\bbP_N G$ (to see this use
$\int_\Omega G p^N\dd x = \int_\Omega G^Np^N\dd x$). Hence, the discretized
gradient-flow equation is given by 
\begin{equation}
  \label{eq:DiscrODE}
  \dot p^N = - \bbK_N(p^N)\rmD\calE_N(p^N) = - \bbK_N(p^N)\big(W'(p^N)- G^N \big) 
\end{equation}
is an $(N {-} 1)$-dimensional ODE. Moreover, by construction we see that a solution
$t \mapsto p^N(t)$ of this ODE even solves the full equation
\eqref{eq:I.BasicEq} if we replace $G$ by $G^N = \bbP_N G$.

\subsection{Wellposedness of the ODE}

Our regularity assumptions on $k$ and $W$ are 
\begin{subequations}
  \label{eq:k.W.conds}
\begin{equation}
    \label{eq:k.W.smooth}
    k\in \rmC^0({[0,\infty[}, {[0,\infty[} ) \quad \text{and} \quad W\in 
\rmC^2({]0,\infty[} ,\R). 
\end{equation}
For $W$ we also impose the standard coercivity conditions
\begin{equation}
    \label{eq:W.coercive}
    W(p)\to \infty \text{ for } p \to \infty, \quad 
    W(p)\to \infty \text{ for } p \searrow 0,
\end{equation}
as well as a a technical assumption for showing well-preparedness of
initial conditions: 
\begin{equation}
  \label{eq:W.wellprepared}
  W = W_1 + W_2, \quad
  W_1 : \; ]0, \infty[ \; \to \bbR \text{ convex,} \quad 
  W_2 \in \rmC^0(]0, \infty[; \bbR) \text{ bounded.}
\end{equation}

The next two conditions relate to the geodesic $\Lambda$-convexity on
sublevels:
\begin{equation}
  \label{eq:Cond.lambda.cvx}
  \begin{aligned}
    \exists\: \lambda_W, &C_k, B_1,B_2,B_3 \in \bbR \ \ \forall\: p>0:
   \\
    &k(p)W''(p)+\frac{k'(p)}2 W'(p)\geq \lambda_W, \ \ |k'(p)|\leq C_k,
   \\
    &\big| k(p) W'(p) \big| \leq B_1 W(p) + B_2+ B_3(p{-}1).
  \end{aligned}
\end{equation}
\end{subequations}
\begin{equation}
  \label{eq:W.doubling}
  \begin{gathered}
    \exists \: C_1, C_2 \geq 0 \ \forall \: p_0, p_1 > 0 \ \forall\: p \in
    \bigl[\min\{p_0,p_1\}, 2 \max\{p_0,p_1\}\bigr] : \\
    W(p) \leq C_1 \bigl(W(p_0) + W(p_1)\bigr) + C_2.
  \end{gathered}
\end{equation}
Because of \eqref{eq:k(p)bounds}, the last condition in
\eqref{eq:Cond.lambda.cvx} can be seen as the one-dimensional version of the
celebrated multiplicative stress controls
$\bigl\lvert \pl_F W(F)F^\top \bigr\rvert \leq CW(F) {+} C^2$ which was
introduced in \cite[Eqn.\ (4)]{Ball84MELE}, see also
\cite{Ball02SOPE,KnZaMi10CGPM,MiRoSa18GERV} for further discussion and
applications. Assumption \eqref{eq:W.doubling} in particular provides a
sufficient control on the energy for the Bhattacharya case $k(p) = \kappa p$,
which will be used in Sections \ref{su:Lipschitz} and
\ref{su:SublevelEVI} only.\medskip 

The following lemma will be used to show global existence for the solutions
$p^N$. In contrast to most other places, the estimate here does depend
crucially on $N\in \N$. 

\begin{lemma}[Sublevel positivity for $\calE_N$]
\label{le:SublevPositive} 
Let $W: {]0,\infty[} \to \R$ be continuous and let the coercivity assumption
\eqref{eq:W.coercive} hold. Then, for all $N \in \bbN$ and $E > \inf \calE$,
there exists $\delta = \delta_{N,E} > 0$ such that for all $p^N \in \bfP_N$,
  \begin{equation*}
    \calE_N(p^N) \leq E \implies p^N(x) \geq \delta_{N,E} \quad \text{ for a.e. } x \in \Omega.
  \end{equation*}
\end{lemma}
\begin{proof}
By hypothesis on $W$, for all $E > \inf\calE $, there exists $\eta_{E} > 0$ such
that $W(p) \leq E$ implies $\eta_{E} \leq p \leq \eta_{E}^{-1}$.

Since $G$ is continuous on $\overline{\Omega}$, we have $\sup{G} < \infty$. Let
$E > \inf{\calE}$, and define $E' := \max\{E, E {+} \sup{G}\}$. Let $N \in \bbN$
and $p^N \in \bfP_N$ be such that $\calE_N(p^N) \leq E$. Then we have
\begin{equation*}
  \begin{aligned}
    \frac{1}{N}\sum_{i=1}^N W(p^N_i) &\leq \frac{1}{N} \sum_{i=1}^N \Bigl(
    W(p^N_i) + (\sup{G} - G^N_i) p^N_i
    \Bigr) \\
    &= \calE_N(p^N) + \sup{G} \leq E'.
  \end{aligned}
\end{equation*}
This implies $W(p^N_i) \leq NE'$ for all $i = 1, \dots, N$. Thus, with
$\eta_{E}$ as in the first paragraph of the proof, we may take
$\delta_{N,E} := \eta_{NE'}$ to conclude.
\end{proof}

With this we are able to conclude that solutions $p^N$ starting with finite
energy are unique and exist globally.

\begin{proposition}[Global existence and uniqueness for fixed $N$]
\label{pr:ODE.GlobExi} 
Let assumptions \eqref{eq:k.W.smooth} and \eqref{eq:W.coercive} hold. Then,
for each $p_0^N\in \dom(\calE_N) :=\bigset{ p^N\in \bfP_N }{
  \calE_N(p^N)<\infty} $ there exists a unique
global solution $p^N: {[0,\infty[} \to \bfP_N$ to
\eqref{eq:DiscrODE} with $p^N(0)=p_0^N$. This defines a global semiflow
$\big(S^N_t)_{t\geq 0}$ on $\dom(\calE)$ which satisfies $\calE_N
\big( S^N_t(p^N_0) \big) \leq \calE_N(p_0^N)$ for all $t>0$ and $ p_0^N \in
\dom(\calE_N) $. 
\end{proposition}
\begin{proof}
From \eqref{eq:k.W.smooth} we see that the ODE for $p^N \in \bfP_N $ has a
$\rmC^1$ right-hand side as long as $p^N_i > 0$ for all $i$. Thus for solutions
starting with $  E_0:= \calE_N(p^N(0))<\infty$ we have a  unique local solution
${[0,t_*[} \ni  t\mapsto p^N(t) $ that  exists as long as  $p^N_i(t)>0$ for all
$i$.

However, along the local solutions, the discrete energy $\calE_N$ is a Lyapunov
function, because we have a gradient-flow equation. Thus, using
$\calE_N(p^N(t)) \leq \calE_N(p^N(0) ) = E_0 < \infty$, sublevel positivity
from Lemma~\ref{le:SublevPositive} gives $p^N_i(t)\geq \delta_{N,E_0} $ for all
$i$ and all $t\in {[0,t_*[}$. Hence, the solution must exist globally.

Uniqueness follows from local Lipschitz continuity. 
\end{proof}

Later on, we will construct solutions $p$ for the full problem
\eqref{eq:I.BasicEq} via approximations $p^N$ where we choose $p^N(0)=\bbP_N
p(0)$. In this process, it will be useful to control the initial energy as
well in the sense of well-prepared initial conditions, see e.g.\
\cite[Eqn.\,(5)]{Serf11GCGF} or \cite[Eqn.\,(5.14)]{Miel23IAGS}. 

\begin{lemma}[Convergence of energies]
\label{le:WellPrepIC}
If \eqref{eq:W.wellprepared} holds, then, for all $p \in \bfP$, we have
\begin{equation*}
    \calE(p) < \infty \quad \Longrightarrow \quad 
        \lim_{N \to \infty} \calE_N(\bbP_N{p}) = \calE(p).
\end{equation*}
\end{lemma}
\begin{proof}
Let $W_1, W_2 : \; ]0, \infty[ \; \to \bbR$ be as in \eqref{eq:W.wellprepared}. 
Then, convexity of $W_1$ and Jensen's inequality imply 
$\int_\Omega W_1(\bbP_N p) \dd x \leq \int_\Omega W_1(p)\dd x$. Moreover,
$\bbP_N p \to p$ in $\rmL^1(\Omega)$ implies $\liminf_{N\to \infty} \int_\Omega
W_1(\bbP_N p) \geq  \int_\Omega W_1(p)\dd x$. Thus, $\int_\Omega W_1(\bbP_N p) 
\dd x \to  \int_\Omega W_1(p)\dd x$ as $N\to \infty$. 

Since  $W_2$ is continuous and $\bbP_N{p} \to p$ a.e., we have 
$W_2(\bbP_N p) \to W_2(p)$ a.e., and the convergence $\int_\Omega W_2(\bbP_N p)
\dd x \to  \int_\Omega W_1(p)\dd x$ follows by the dominated convergence theorem. 

With $G \in \rmL^\infty(\Omega)$, we obtain  that
\begin{equation*}
  \begin{aligned}
    \lim_{N \to \infty} \calE_N(\bbP_N{p})
      &= \lim_{N \to \infty} \int_\Omega 
        \bigl( W_1(\bbP_N{p}) + W_2(\bbP_N{p}) - G \bbP_N{p} \bigr) \dd x
     \\
      &= \int_\Omega \bigl( W_1(p) + W_2(p) - G p \bigr) \dd x 
       \ = \ \calE(p),
  \end{aligned}
\end{equation*}
as required.
\end{proof}

\subsection{Infinitesimal stretching of the metric along solutions}
\label{su:InfinitesimalStretch}

By the classical ODE theory and the smoothness from \eqref{eq:k.W.smooth}, we
can differentiate with respect to the initial condition and a change in the
loading $G^N$. Denoting by
\[
  y^N \in \bfT_N:= \biggl\{
    y^N = \sum_{i=1}^n  y^N_i \bbone_{{[ \tau^N_{i-1}, \tau^N_i [}} \biggm|
    \int_\Omega y^N(x) \, \rmd{x} =0
  \biggr\} \subset \rmL^1(\Omega) 
\]
the tangential directions at any $p^N \in \bfP_N$ we obtain 
the coupled ODE system
\begin{subequations}
  \label{eq:ODE.Tang}
  \begin{align}
    \dot p^N(t) &= \bm V_N(p^N(t)), \\
    \dot y^N(t) &=  \rmD \bm V_N(p^N(t)) y^N(t) + \bbK_N(p^N(t)) \zeta^N,
  \end{align}
\end{subequations}
where the vector field $\bm{V}_N : \bfP_N \to \bfT_N$ is given by
$ \bm{V}_N(p^N) := -\bbK_N(p^N)\rmD\calE_N(p^N) $ and $\zeta^N \in \bfT_N^*$
denotes a variation in $G^N$ needed to compare solutions between systems with
different $N$, e.g.\ for $N=2M$ one can use $\zeta^N = G^{2M}-G^M$. Here,
$\bfT_N^*$ denotes the dual space of the vector space $\bfT_N$, which we can
naturally identify with arbitrary step functions
$\zeta^N = \sum_{i=1}^N \zeta^N_i {\bf1}_{[\tau^N_{i-1}, \tau^N_i[}$ via the
$\rmL^2$ inner product.

The length of the vectors $y^N$ is measured by the metric induced by $k$, namely
\[
  \calR_N(p^N,y^N) := \frac12 \langle \bbG_N(p^N) y^N, y^N\rangle
  = \frac12 \sum_{i=1}^N  \frac1N\, \frac{(y^N_i)^2}{k(p^N_i)}. 
\]
Note that $\bbK(p^N)\bbG(p^N)y^N = y^N$ for all
$(p^N, y^N) \in \bfP_N \times \bfT_N$. In the finite-dimensional context, the
quadratic form $\calR^*$ becomes
\begin{equation*}
  \calR^*_N(p^N, \xi^N)
  := \frac{1}{2} \langle \xi^N, \bbK_N(p^N)\xi^N \rangle 
  = \frac{1}{2} \sum_{i=1}^N \frac{1}{N} k(p_i^N) \biggl(
    \xi^N_i - \frac{[k_{p^N} \xi^N]}{[k_{p^N}]}
  \biggr)^2
\end{equation*}
The rate of change of $\calR_N$ or $\calR_N^*$ along solutions of
\eqref{eq:ODE.Tang} is given through the second derivative of $\calE_N$, namely
the differential geometric Hessian $ \wt\bbH_N:\bfT_N\to \bfT_N^* $. This also
involves the derivative of the metric $\bbG_N$ or the Onsager operator
$\bbK_N$. We call $\bbH_N = \bbK_N \wt\bbH_N \bbK_N$ the contravariant Hessian
of $\calE_N$ with respect to the metric $\bbG_N$.

\begin{theorem}[Stretching of the tangent vectors]
\label{th:Stretching}
Let $(p^N, y^N) : {[0,\infty[} \to \bfP_N \ti \bfT_N $ be a solution
of \eqref{eq:ODE.Tang}. Then we have the relation
\begin{align}
\nonumber
  \frac\rmd{\rmd t} \calR_N  (p^N(t),y^N(t)) 
  &=  -  \bigl\langle  \wt\bbH_N(p^N(t))y^N(t),y^N(t) \bigr\rangle 
         + \langle y^N(t),\zeta^N  \rangle \\
\nonumber
  &=  - \bigl\langle   \bbH_N(p^N(t))\bbG_N(p^N(t))y^N(t), 
         \bbG_N(p^N(t))y^N(t)\bigr\rangle 
          + \langle y^N(t),\zeta^N \rangle  \\
\label{eq:StretchRelat}
  &  = \frac{\rmd}{\rmd{t}} \calR_N^*\bigl(p^N(t), \bbG_N(p^N(t))y^N(t)\bigr), 
\\
\nonumber
 \text{where } &\big\langle \xi, \bbH_N(p) \xi \big\rangle =
  \int_\Omega H_N[p](x) \biggl( \xi - \frac{[k_p\xi]}{[k_p]} \biggr)^2 \dd x 
  \\
\nonumber
  \text{and } H_N[p] &= k(p) \biggl(
  k(p)W''(p) 
  + \frac12 k'(p) \biggl(
  W'(p) - G^N - \frac{[k_p(W' {-} G^N)]}{[k_p]}
  \biggr)  
  \biggr).
  \end{align}
\end{theorem} 
\begin{proof}
Define $\xi^N(t) := \bbG_N(p^N(t)) y^N(t)$ for $t > 0$. Then
\begin{equation*}
  \dot{\xi}^N 
  = \bigl(\rmD\bbG_N(p^N)\dot{p}^N\bigr)y^N + \bbG_N(p^N)\dot{y}^N.
\end{equation*}
Using the property that $\bbK_N(p^N)\bbG_N(p^N) = \id_{\bfT_N}$ and
\begin{equation*}
  0 
  = \frac{\rmd}{\rmd{t}}\bbK_N(p^N) \bbG_N(p^N)
  = \bigl(\rmD\bbK_N(p^N)\dot{p}^N \bigr) \bbG_N(p^N)
     + \bbK_N(p^N) \bigl( \rmD\bbG_N(p^N)\dot{p}^N\bigr),
\end{equation*}
we therefore have
\begin{equation*}
  \begin{aligned}
    \bbK_N(p^N)\dot{\xi}^N
    &= -\bigl( \rmD\bbK_N(p^N) \dot{p}^N \bigr) \bbG_N(p^N) y^N + \dot{y}^N \\
    &= -\bigl( \rmD\bbK_N(p^N) \bm{V}_N(p^N) \bigr) \xi^N + \rmD\bm{V}_N(p^N)
    \bbK_N(p^N)\xi^N  +  \bbK_N(p^N) \zeta^N.
  \end{aligned}
\end{equation*}
Differentiating $\calR_N^*(p^N, \xi^N)$ then gives us
\begin{equation*}
 \begin{aligned}
  &\frac{\rmd}{\rmd{t}}\calR^*_N(p^N, \xi^N)
   = \frac{\rmd}{\rmd{t}}\frac{1}{2}\langle \xi^N, \bbK_N(p^N)\xi^N \rangle 
    %% \\ &
   = \frac{1}{2} \bigl\langle \xi^N, \bigl(\rmD\bbK_N(p^N)\dot{p}^N\bigr)
   \xi^N \bigr\rangle + \langle \xi^N, \bbK_N(p^N)\dot{\xi}^N \rangle \\
   &\quad = -\frac{1}{2} \bigl\langle \xi^N,
   \bigl(\rmD\bbK_N(p^N)\bm{V}_N(p^N)\bigr)\xi^N \bigr\rangle + \bigl\langle
   \xi^N, \rmD\bm{V}_N(p^N)\bbK_N(p^N)\xi^N \bigr\rangle + \bigl\langle
   \xi^N, \bbK_N(p^N) \zeta^N \bigr\rangle.
 \end{aligned}
\end{equation*}

To calculate the quadratic form in
$\xi^N$ on the right hand side of the above equation, we consider the
case $\zeta^N =0$ first. We abbreviate our notation by taking $\xi =
\xi^N$, dropping the explicit dependence on
$p^N$, and using the averaging notation
$[\cdot]$ extensively. Moreover, we incorporate the loading
$G^N$ into the energy density $W$ by defining
  \begin{equation*}
    \widetilde{W}(x, p) := W(p) - G^N(x)p.
  \end{equation*}
  Derivatives $\wt{W}'$, $\wt{W}''$ are taken with respect to $p$.

  Firstly, we have
  \begin{equation*}
    \begin{aligned}
      \left\langle \xi^N, \rmD{\bm{V}_N}(p^N)\xi^N \right\rangle 
      &= -\biggl[
        k\xi\biggl(
          \xi - \frac{[k\xi]}{[k]}
        \biggr) \biggl(
          k' \biggl(
            \wt{W}' - \frac{[k\wt{W}']}{[k]}
          \biggr) + k\wt{W}''
        \biggr)
      \biggr] \\
      &\hspace{-2em} + \frac{[k\xi]}{[k]} \biggl(
        \biggl[
          (k'\wt{W}' + k\wt{W}'') k \biggl(
            \xi - \frac{[k\xi]}{[k]}
          \biggr)
        \biggr]
        - \frac{[k\wt{W}']}{[k]} \biggl[
          k'k \biggl(
            \xi - \frac{[k\xi]}{[k]}
          \biggr)
        \biggr]
      \biggr) \\
      &= -[k^2\xi^2\wt{W}''] + 2 \frac{[k\xi]}{[k]}[k^2\xi\wt{W}''] - \frac{[k\xi]^2}{[k]^2}[k^2\wt{W}''] \\
      &\hspace{1em} - [k'k\xi^2\wt{W}'] + \frac{[k'k\xi^2]}{[k]}[k\wt{W}']  + 2\frac{[k\xi]}{[k]}[k'k\xi\wt{W}'] 
        - 2\frac{[k'k\xi][k\xi]}{[k]^2}[k\wt{W}'] \\
      &\hspace{1em} - \frac{[k\xi]^2}{[k]^2}[k'k\wt{W}'] + \frac{[k'k][k\xi]^2}{[k]^3}[k\wt{W}'] \\
      &= - \int_\Omega k \biggl(
        k\wt{W}'' + k' \biggl(
          \wt{W}' - \frac{[k\wt{W}']}{[k]}
        \biggr) 
      \biggr) \biggl(
        \xi - \frac{[k\xi]}{[k]}
      \biggr)^2 \, \rmd{x}.
    \end{aligned}
  \end{equation*}

  Secondly, we have
  \begin{equation*}
    \begin{aligned}
      \bigl\langle \xi^N, \bigl(\rmD\bbK_N(p^N)\bm{V}_N(p^N)\bigr)\xi^N \bigr\rangle
      &= \biggl[ -k'k \biggl(
          \wt{W}' - \frac{[k\wt{W}']}{[k]}
        \biggr)\biggl(
          \xi - \frac{[k\xi]}{[k]}
        \biggr) \xi
      \biggr] \\
      &\hspace{-2em} + \frac{[k\xi]}{[k]} \biggl(
        \biggl[
          k'k \biggl(
            \wt{W}' - \frac{[k\wt{W}']}{[k]}
          \biggr) \xi
        \biggr]
        - \frac{[k\xi]}{[k]} \biggl[
          k'k \biggl(
            \wt{W}' - \frac{[k\wt{W}']}{[k]}
          \biggr)
        \biggr]
      \biggr) \\
      &= - \int_\Omega k'k \biggl(
        \wt{W}' - \frac{[k\wt{W}']}{[k]}
      \biggr) \biggl(
        \xi - \frac{[k\xi]}{[k]}
      \biggr)^2 \, \rmd{x}.
    \end{aligned}
  \end{equation*}

Noting $\wt{W}'' = W''$, it follows that 
\begin{multline*}
  -\frac{1}{2}\bigl\langle \xi^N, \bigl(\rmD\bbK_N(p^N)\bm{V}_N(p^N)\bigr) 
        \xi^2 \bigr\rangle
   + \langle \xi^N, \rmD\bm{V}_N(p^N)\bbK_N(p^N)\xi^N \rangle \\
 = -\int_\Omega H_N[p^N](x) \biggl(
     \xi^N - \frac{[k_{p^N}\xi^N]}{[k_{p^N}]}
    \biggr) \, \rmd{x},
\end{multline*}
as required. Reinserting $ \langle \xi^N, \bbK_N(p^N) \zeta^N \rangle 
= \langle \zeta, y^N \rangle $ the desired stretching relation
\eqref{eq:StretchRelat} holds, and Theorem \ref{th:Stretching} is established.
\end{proof}

We can now define the infinitesimal stretching rate on the energy sublevels
\[ 
  \Sigma_N(E):= \bigset{p^N \in \bfP_N }{ \calE_N(p^N) \leq E}.
\]

\begin{definition}[Infinitesimal sublevel stretching rate]
\label{de:LocSublevSR} 
Given $E> \inf \calE_N$ we define 
\begin{multline*}
  \qquad\Lambda_N^\mafo{inf}(E) := 
  \sup\Bigset{ \Lambda \in \R }{ 
    \forall \, p^N \in \Sigma_N(E), \ \forall \xi^N \in \bfT_N: \\
    \langle \xi^N,\bbH_N(p^N) \xi^N\rangle \geq \Lambda \langle \xi^N,
    \bbK_N(p^N) \xi^N \rangle 
  }.  \qquad
\end{multline*}
\end{definition}

We now show that under suitable assumptions we always have
$\Lambda_N^\mafo{inf}(E) > -\infty$. For this, for an arbitrary
$p^N \in \bfP_N$ we first estimate the the optimal $\wh\Lambda(p^N)$ in the
positive semi-definiteness estimate
$\bbH_N(p^N) - \wh\Lambda(p^N)\,\bbK_N(p^N) \geq 0$.

\begin{lemma}[Lower bound for Hessian] 
\label{le:LambdaIf.LowBd}
For $p^N\in \bfP_N$ we have the lower bound 
\begin{align*} 
   &\big\langle \xi, \bbH_N(p^N) \xi \big \rangle \
      \geq \ \wh\Lambda(p^N)\,  \big\langle \xi, \bbK_N(p^N)\xi \big\rangle 
      \  \text{ for all } \xi \in \bfT_N^*, 
      \; \text{ where }
  \\[0.4em]
   &  \wh\Lambda(p^N) := \inf_{ \genfrac{}{}{0pt}{2}{q>0}{x\in \Omega}} 
      \biggl( k(q)W''(q) + \frac{k'(q)}2
        \bigl(W'(q) {-} G^N(x)\bigr)
      \biggr) 
      - \frac{[k_{p^N}(W'(p^N) {-} G^N)]}{2\,[k_{p^N}]}\,
      \sup_{q>0} k'(q).
\end{align*}
\end{lemma}
\begin{proof}
  We simply observe that
  $\langle \xi, \bbK_N(p^N) \xi \rangle = \int_\Omega k(p^N)\Bigl(\xi -
  \frac{[k_{p^N}\xi]}{[k_{p^N}]} \Bigr)^2 \dd x$ and recall the corresponding
  formula for $\bbH_N$ from \eqref{eq:StretchRelat} involving $H_N[p^N]$. Hence, the desired lower
  bound is found by observing that $H_N[p^N] \geq
  \wh{\Lambda}(p^N)k(p^N)$.
\end{proof}

Relying on the assumption \eqref{eq:Cond.lambda.cvx} we are now ready to give a
lower bound on $\Lambda_N^\mafo{inf}(E)$ that is independent of $N$ by an
infinitesimal sublevel stretching bound $\bfL^\mathrm{inf}(E)$.

\begin{corollary}[Lower bound on $\Lambda_N^\mafo{inf}(E)$]
\label{co:LowBddLambda}
Assuming \eqref{eq:k(p)bounds} and \eqref{eq:k.W.conds} hold, we have the lower
bound
\[
\Lambda_N^\mafo{inf}(E) 
\geq \bfL^\mathrm{inf}(E) 
:= \lambda_W - \frac{C_k}{2} 
    \biggl(  \Bigl( 1 + \tdfrac{\,\ol{\kappa}\,}{\ul{\kappa}} \Bigr) 
             \left\lVert G \right\rVert_{\rmL^\infty}
               + \tdfrac{1}{\ul{\kappa}} \bigl( B_1 E + B_2 \bigr) 
    \biggr)
    > -\infty.
\]
where the constants $\lambda_W, C_k, B_1$, and $B_2$ are defined in
\eqref{eq:Cond.lambda.cvx} and $\ol \kappa$ and $\ul\kappa$ in
\eqref{eq:k(p)bounds}.
\end{corollary}

With the above preparations we are now able to provide estimates independent of
$N$ for the growth of the tangent vectors $y^N$ satisfying \eqref{eq:ODE.Tang}.
For this we combine Theorem \ref{th:Stretching} and Corollary
\ref{co:LowBddLambda} with Gr\"onwall's lemma. For the norm induced by $\bbG_N$
on $\bfT_N$ we use the shorthand notation 
\[
 \norm{ y^N}{\bbG_N} :=  \bigl\langle y^N, \bbG_N(p^N)y^N \bigr\rangle^{1/2} . 
\]

\begin{proposition}[Growth estimate for $y^N$]
\label{pr:Growth.yN}
 Assuming \eqref{eq:k(p)bounds} and \eqref{eq:k.W.conds}, the
 solutions $(p^N,y^N)$  with $p^N(0) \in \Sigma_N(E)$
satisfy the growth estimate 
\begin{equation}
  \label{eq:Growth.yN}
 \begin{aligned} 
 \norm{ y^N(t)}{\bbG_N} 
  &\leq \ee^{ - \bfL^\mathrm{inf}(E)  \,t }   \,
      \norm{y^N(0)}{\bbG_N}  + \mathsf M_{ \bfL^\mathrm{inf}(E) } (t)\, 
   \sqrt{\,\ol\kappa\,} \,\|\zeta^N \|_{\rmL^\infty} \quad \text{for } t>0,
\end{aligned}
\end{equation}
where $\mathsf M_\lambda(t) := \int_0^t \ee^{-\lambda s} \dd s$. 
\end{proposition}
\begin{proof} We first observe that $ p^N(0) \in \Sigma_N(E)$ implies
  $p^N(t)\in \Sigma_N(E)$ for all $t>0$, see Proposition \ref{pr:ODE.GlobExi}.  
Next we use $\calR_N(p^N,y^N)=\frac12\norm{y^N}{\bbG_N}^2$ with
\eqref{eq:StretchRelat} and Corollary~\ref{co:LowBddLambda} to obtain 
\[
\frac\rmd{\rmd t} \norm{y^N(t)}{\bbG_N}^2 \leq -2 \, \bfL^\mathrm{inf}(E) \,
\norm{y^N(t)}{\bbG_N}^2 + 2\,\big\langle y^N(t), \zeta^N\big\rangle. 
\]
The last term can be estimated using $\bbG_N\bbK_N=\mafo{id}_{\bfP_N}$ and the
Cauchy-Schwarz inequality as follows
\begin{align*}
\big\langle y^N, \zeta^N\big\rangle 
&\leq \norm{y^N}{\bbG_N} \big\langle \zeta^N, \bbK_N(p^N)\zeta^N\big\rangle^{1/2} 
 =  \norm{y^N}{\bbG_N} \biggl(\int_\Omega k(p^N) 
      \Big(\zeta^N-\tdfrac{[k_{p^N}\zeta^N]}{ [k_{p^N}]} \Big)^2 \dd x\biggr)^{1/2} 
\\
&\leq  \norm{y^N}{\bbG_N} \biggl(\int_\Omega \ol\kappa \,p^N \, 
   \big(\zeta^N)^2  \dd x\biggr)^{1/2} \leq \norm{y^N}{\bbG_N} \,\varrho \quad
   \text{ with } \varrho:= \sqrt{\,\ol\kappa\,}\,\|\zeta^N \|_{\rmL^\infty}.
\end{align*} 
 
For $\delta>0$ we now set $n(t)=\big(
\norm{y^N(t)}{\bbG_N}^2+\delta\big)^{1/2}$ and find 
\[
2n \dot n = \tfrac\rmd{\rmd t} (n^2) = \tfrac\rmd{\rmd t} 
 \norm{y^N(t)}{\bbG_N}^2 \leq - 2\,\bfL^\mathrm{inf}(E) \,n^2 + 2\delta
  \, \bfL^\mathrm{inf}(E) + 2 n \varrho.
\]
Because of $n\geq \delta^{1/2} >0$ we can divide by $2 n$ and obtain $\dot n
\leq - \bfL^\mathrm{inf}(E)  \,n + \delta^{1/2} |  \bfL^\mathrm{inf}(E) | +
\varrho$, and Gr\"onwall's lemma yields 
\[
\norm{y^N(t)}{\bbG_N} \leq n(t) \leq \ee^{- \bfL^\mathrm{inf}(E) \,t}  n(0) +
\mathsf M_{\bfL^\mathrm{inf}(E) } (t) \big( \delta^{1/2} | \bfL^\mathrm{inf}(E)| +
\varrho \big) \quad \text{for all } t>0.
\] 
Using $n(0)=( \norm{y^N(0)}{\bbG_N}{+}\delta )^{1/2}$ the limit $\delta\to
0$ gives the desired result. 
\end{proof}

To compare two solutions $p^N(t)$ and $\wt p^N(t)$ starting at different
initial data $p^N(0) \neq \wt p^N(0)$ or to compare solutions $p^N(t)$ and
$p^M(t)$  with $N \neq M$ we need to derive estimates using the geodesic
distance $\calD^k_N$  induced by the Riemannian metric tensor $\bbG_N$. This
will be the focus of the next section.

\section{Geodesic distances} 
\label{se:GeodDist}

Before we derive uniform contraction and Lipschitz estimates for the flows
constructed above, we have to introduce the distances associated with the
Riemannian metric tensor $\bbG_N$ in $\bbP_N$. In Section \ref{se:Def.Bhatt} we
first study the 
classical Bhattacharya distance (also called spherical Hellinger distance),
cf.\ \cite{Miel25?NHDV} and then in Section \ref{se:Dist.DkN} introduce the
distance $ \calD^k_N$. Section \ref{su:Def.calD.k} discusses the limiting
distance $\calD^k$ on $\bfP$. 

\subsection{The  Bhattacharya distance} 
\label{se:Def.Bhatt}

The case $k^*(p)=4p$ gives the Bhattacharya distance $\Bh=\calD^{k^*}$. 
For $p_0,p_1\in \bfP$ we have  
\begin{align*}
\He(p_0,p_1)&= \bigl\lVert \sqrt{p_1} - \sqrt{p_0} \bigr\rVert_{\rmL^2(\Omega)}\in
[0,\sqrt{2}] \quad \text{and}  
\\
 \Bh (p_0,p_1)&= 2 \arcsin\biggl( \frac12\He(p_0,p_1)\biggr) =
\arccos\biggl( 1- \frac12\He(p_0,p_1)^2\biggr) \in
[0,\pi/2]  .
\end{align*}
The Hellinger distance $\He$ is a geodesic distance on $\rmL^1_\geq(\Omega)$
and the Bhattacharya distance $\Bh$ is the induced Fisher-Rao distance on the
closed subset $\bfP\subset \rmL^1_\geq(\Omega)$ of probability densities, see
\cite{ Miel25?NHDV}. For our theory it will be important that the geodesic
curves in $(\bfP,\Bh)$ are given explicitly in \cite{LasMie19GPCA,
  Miel25?NHDV}, namely 
\begin{align*}
\gamma^\Bh(s,x)&= \wt n(t(s)) \Big( (1{-}t(s))\sqrt{p_0(x)} + t(s)\sqrt{p_1(x)}
\Big)^2 \\
&\text{with }t(s) = \frac{\sin(s\delta)}{\sin(s\delta)+\sin(\delta{-}s\delta)},
 \quad \delta = \Bh(p_0,p_1), 
\\
&\text{and } \wt n(t) =    \frac1{1-2(t{-}t^2)(1{-}\cos\delta)}.
\end{align*} 

For our analysis it will be very helpful that geodesics between piecewise
constant densities in $\bfP_N$ stay in $\bfP_N$, which is easily checked by the
explicit formula.

\begin{lemma}[$\bfP_N$ is geodesically complete]
\label{le:Bhatt.PN.Complete}
For all $p_0,p_1\in \bfP_N$ the unique Bhattacharya geodesic $s\mapsto
\gamma^\Bh(s) \in \bfP$ stays inside of $\bfP_N$. 
\end{lemma}

The following lemma will be useful to show that Bhattacharya geodesics
connecting $p_0$ and $p_1$ from a sublevel $\Sigma(E)$ lie in a sublevel
$ \Sigma(\wt E) $ for a suitable $ \wt E \geq E $. 

\begin{lemma}[Density bounds for geodesics in $(\bfP,\Bh)$]
\label{le:DensBhatt}
For all $p_0,p_1\in \bfP$ the Bhattacharya geodesic satisfies the estimate 
\begin{equation}
  \label{eq:EstBhGeo}
\forall\, s\in [0,1]: \quad  \min\{p_0(x),p_1(x)\}\leq  \gamma^\Bh(s,x) \leq 2
\max\big\{ p_0(x),p_1(x)\big\}  \quad \text{a.e.\  in } \Omega.
\end{equation}
\end{lemma}
\begin{proof} Using $\delta=\Bh(p_0,p_1) \in [0,\pi/2]$ we easily see $\wt n(t)
  \in [1,2]$. With this and $\min\left\{ p_0, p_1 \right\}^\frac{1}{2} \leq \sqrt{p_j}
  \leq \max\left\{p_0,p_1\right\}^\frac{1}{2}$, the desired result follows. 
\end{proof}

The topology induced by $\He$ or $\Bh$ on $\bfP$ is the same as the strong
topology of $\rmL^1(\Omega)$ because of the well-known estimates (see e.g.\
\cite{Miel25?NHDV})
\begin{equation}
  \label{eq:Bh.He.L1}
  \frac8{\pi^2} \,\Bh(p,\wt p)^2 \leq \He(p,\wt p)^2 \leq \| p{-}\wt
p\|_{\rmL^1(\Omega)} \leq 2\, \He(p,\wt p) \leq 2  \, \Bh(p,\wt p).  
\end{equation}

\subsection{The distance \texorpdfstring{$\calD^k_N$}{DkN} on 
    \texorpdfstring{$\bfP_N$}{PN}}
\label{se:Dist.DkN}

To combine the knowledge on the Bhattacharya distance with the stretching
properties of the previous section, we denote the Riemannian distance on
$(\bfP_N, \bbG_N )$ by the symbol $\calD^k_N$. In the case $k^*(p)=4p$ we have
$\calD^{k^*}_N=\Bh|_{\bfP_N}$. Below, it will be important that $k$ satisfies
the upper and lower bounds in \eqref{eq:k(p)bounds}, which means that we are
able to compare with $\Bh$. For $ p_0^N, p_1^N \in \bfP_N$ the distance
$\calD^k_N(p_0^N, p_1^N) $ is defined via 
\begin{multline}
\label{eq:Def.Metric.N}
\quad  \calD^k_N(p_0^N, p_1^N)^2 
  := \inf\biggl\{\int_0^1 \hspace{-0.5em} \int_\Omega
      \frac{\pl_sp^N(s,x)^2}{k(p^N(s,x))} 
    \dd x \dd s 
    \biggm| p^N \in \rmC^1([0,1]; \bfP_N), 
 \\
    \ p^N(0) = p_0^N,\ p^N(1) = p_1^N
  \biggr\} . \quad
\end{multline}
The minimizers are the so-called (constant-speed) geodesic curves connecting
the points $p_0^N$ and $p_1^N$. The interior integral in
\eqref{eq:Def.Metric.N} is precisely $\langle \pl_s{p^N}, 
\bbG_N(p^N)\pl_s{p^N} \rangle$. An alternative formulation using $\bbK_N$ and
the continuity equation $\pl_s p^N = \bbK_N(p^N)\xi^N $ reads 
\begin{multline*}
  \calD^k_N(p_0^N, p_1^N)^2 = \inf\biggl\{ 
    \int_0^1 \langle \xi^N, \bbK_N(p^N) \xi^N \rangle \, \rmd{s}
    \biggm| (p^N, \xi^N) \in \rmC^1([0, 1], \bfP_N \times \bfT_N^*), \\ \pl_s p^N = \bbK_N(p^N)\xi^N, \;
    p^N(0) = p_0^N, \ p^N(1) = p_1^N  
  \biggr\}.
\end{multline*}
Note that geodesics exist by compactness of $\bfP_N$.

Since $p\mapsto k(p)$ has an upper and a lower linear bound, we can compare
$\calD^k_N$ with $\Bh$ as follows.

\begin{lemma}[Comparing $\calD_N^k$ and $\Bh$]
\label{le:CompareDist} 
If $k$ satisfies the assumption \eqref{eq:k(p)bounds}, then  the
distances $\calD^k_N$ and $\Bh$ are equivalent on $\bfP_N$ and satisfy the
estimate
\[
C_\mafo{low} \, \Bh(p^N_0,p^N_1) \leq \calD_N^k(p^N_0,p^N_1) \leq
  C_\mafo{upp}  \,\Bh(p^N_0,p^N_1)  \quad \text{for all }p^N_0,p^N_1\in \bfP_N.
\]
with $C_\mafo{low}= 2/\sqrt{\ol\kappa}$  and $C_\mafo{upp}=
2 / \sqrt{\ul\kappa}$. 
\end{lemma}

Without further properties we cannot expect that $\calD^k_N$ has as good
properties as $\Bh$, namely that $\bfP_N$ is geodesically complete in
$\big(\calP_{NM}, \calD^k_{NM}\big)$ for all $M\in \N$. In particular, one has
to be aware that in general  
geodesics in $(\bfP_N,\calD^k_N)$  may not be unique and we may have
\[
\exists\: N,M\in \N\ \exists\: p_0^N, p_1^N \in \bfP_N: \quad 
\calD^k_{MN}(p_0^N, p_1^N) \lneqq \calD^k_N(p_0,p_1).
\]
See Appendix \ref{se:calD.Counterexample} for an example with $N=2$ and $M
\gg 1$, but we expect there should already be counterexamples with $N=M=2$.

\subsection{The distance \texorpdfstring{$\calD^k$}{Dk} on the infinite dimensional
    space \texorpdfstring{$\bfP$}{P} }
\label{su:Def.calD.k}

Using the above uniform estimates with respect to $\Bh$ allow us 
to define a distance $\calD^k$ on $\bfP\ti\bfP$ as follows:
\begin{equation}
  \label{eq:Def.calD.k}
  \calD^k(p_0,p_1):= \lim_{N\to \infty} \Big( \lim_{M\to \infty}
  \calD^k_{M!}\big( \bbP_N p_0, \bbP_N  p_1 \big) \Big) .
\end{equation}
Note that in the inner limit, the value  $\calD^k_{M!}\big( \bbP_N p, \bbP_N
\wt p\big)$ is well-defined for all $M\geq N$ because then $M!$ is a multiple
of $N$. Moreover, because of $\calD^k_{M} \geq \calD^k_{M!}$ we also have 
$\calD^k(\bbP_Np_0, \bbP_N p_1) = \inf_{K\in \N} \calD^k_{NK}(\bbP_Np_0, \bbP_N
p_1) $, but having a limit instead of an infimum is crucial, see
\cite[Exa.\,3.6]{DeSoTa25ICSH} for a related case where a definition via infimum
leads to the failure of the triangle inequality. 

\begin{proposition}[Definition of $\calD^k$]
\label{pr:def.calD.k}
Under assumption \eqref{eq:k(p)bounds}, $\calD^k:\bfP\ti\bfP\to {[0,\infty[}$
is well-defined and provides a distance on $\bfP$ satisfying
\[
C_\mafo{low} \, \Bh(p_0,p_1) \leq \calD^k(p_0,p_1) \leq
  C_\mafo{upp}  \,\Bh(p_0,p_1)  \quad \text{for all }p_0,p_1\in \bfP.
\]
\end{proposition}
\begin{proof}
To study the inner limit we fix $N,L\in \N$ and observe that the sequence  
\[
M\mapsto  \calD^k_{M!}\big( \bbP_N p_0, \bbP_L p_1\big)
\]
is well-defined for $M\geq \max\{L,N\}$ and nonincreasing in $M$ because of 
$\bfP_{\max\{L,N\}}\subset \bfP_{M!}\subset \bfP_{(M{+}1)!}$. Hence, 
the limit $\calA(\bbP_N p_0, \bbP_L p_1)$ exists and can be estimated from above
and below by $\Bh(\bbP_N p_0, \bbP_L p_1)$ using Lemma \ref{le:CompareDist}
with $M!$ in place of $N$. 

Thus, $\calA$ is well-defined on $\wt\bfP \ti \wt \bfP$ with
$\wt\bfP:= \bigcup_{N\in \N} \bfP_N$, which is dense in
$\bfP\subset \rmL^1(\Omega)$, because for all $p_0\in \bfP$ we have
$\bbP_N p_0 \to p_0$ in $\rmL^1(\Omega)$ as $N\to \infty$.  By
\eqref{eq:Bh.He.L1} this implies
$\calA(\bbP_Np_j,\bbP_L p_j)\leq C_\mafo{upp} \Bh(\bbP_Np_j,\bbP_L p_j) \leq C
\|\bbP_Np_j{-}\bbP_L p_j \|_{\rmL^1}^{1/2} \to 0$.  Moreover, $\calA$ inherits
the triangle inequality from $\calD^k_{M!}$ by taking the limit $M\to \infty$.
With this we find
\begin{align*}
&\big| \calA(\bbP_Np_0,\bbP_N p_1)- \calA(\bbP_L p_0,\bbP_L p_1) \big|  \\
&\leq \big| \calA(\bbP_Np_0,\bbP_N p_1)- \calA(\bbP_L p_0,\bbP_N p_1) \big|  
  +\big| \calA(\bbP_Lp_0,\bbP_N p_1)- \calA(\bbP_L p_0,\bbP_L p_1) \big|  
\\ 
& \leq  \calA(\bbP_Np_0,\bbP_L p_0) + \calA(\bbP_N p_1,\bbP_L p_1) \\
%   \\ &
&\leq C \|\bbP_Np_0{-}\bbP_L p_0 \|_{\rmL^1}^{1/2} + \|\bbP_Np_1{-}\bbP_L p_1
\|_{\rmL^1}^{1/2}  \to 0 \quad \text{ as } N, L \to \infty. 
\end{align*}
Hence, the outer limit in \eqref{eq:Def.calD.k} exists, and $\calD^k$ is
well-defined and again satisfies the triangle inequality.
\end{proof}

On all of $\bfP$ we may try to define the distance $\wt\calD^k$ by minimizing
the viscous cost
\begin{equation}
  \label{eq:Def.Metric}
  \wt\calD^k(p_0,p_1)^2 := \inf\biggl\{
    \int_0^1 \hspace{-0.5em} \int_\Omega
    \frac{\pl_sp(s,x)^2}{k(p(s,x))} \dd x \dd s  
    \biggm| p \in \rmC^1([0,1];\bfP), \ p(0)=p_0,\ p(1)=p_1 
  \biggr\}.
\end{equation}
However, one may encounter the nonexistence of geodesics, but there are
still approximate geodesics and $(\bfP,\wt\calD^k)$ might be a length space. 

\begin{remark}[Towards a characterization of $\calD^k$] 
It is expected that $\calD^k= \wt \calD^k$, but we are not able to show
this. If $p\mapsto k(p)$ is concave, then the functional defined in
\eqref{eq:Def.Metric} is convex, and a simple application of Jensen's
inequality shows that $\calD^k_{NM}|_{\bfP_N} = \calD^k_N$. Moreover, the existence
of constant-speed geodesics can be shown using convexity. Thus, we conjecture
that $\calD^k= \wt \calD^k$ holds in this case. 

For general $k$, in particular the one in Appendix
\ref{se:calD.Counterexample}, this might still be true, but it remains open
whether $\big(\bfP,\calD^k\big)$ is a geodesic distance or at least a length
space. 

More importantly, it would be necessary to show upper and lower bounds for the
geodesics in $\big(\bfP,\calD^k\big)$ generalizing the estimates in Lemma
\ref{le:DensBhatt}, even in finite dimensions. There the geodesics
$s \mapsto \gamma^N(s)=\big(\gamma_i(s)\big)_{i=1,..,N}$ satisfy the equation
\[
\gamma'_i(s) = k(\gamma_i(s)) \big( \xi_i(s){-}\lambda(s)\big), \quad 
\xi'_i(s) = - \frac12 \,k'(\gamma_i(s)) \big( \xi_i(s){-}\lambda(s)\big)^2 \quad
\]
with $ \lambda(s) = [k(p)\xi]/[k(p)]=\big(\sum_{j=1}^N k(\gamma_j(s))\xi_j(s)\big) \big/
\big( \sum_{j=1}^N k(\gamma_j(s))\big) $. A direct calculation gives
$\lambda' = \frac1{2[k(p)]} \: \sum_{j=1}^N k'(\gamma_j)
k(\gamma_j)\big(\xi_j{-}\lambda\big)^2 \geq 0$, and we find the curvature estimate
\begin{align*}
\gamma''_i&= k'(\gamma_i)\gamma'_i \big(\xi_i{-}\lambda\big) + k(p)\big( \xi'_i -
\lambda'\big) 
 % \\ & 
= \frac{k'(\gamma_i)}{2k(\gamma_i)} \,(\gamma'_i)^2 - k(\gamma_i)\lambda' . 
\end{align*}
Again it is expected that in the case $k'(p)\geq 0$ for all $p$, one can show
the lower estimate $\gamma_i(s) \geq \min\big\{\gamma_i(0),
\gamma_i(1)\big\}$. However, it remains open what conditions are needed to find
$C>1$ such that the upper estimate $\gamma_i(s) \leq C \min\big\{\gamma_i(0),
\gamma_i(1)\big\}$ holds. 
\end{remark}

\section{Local contraction and global Lipschitz continuity}
\label{se:LocContrGlobLip}

\subsection{Local contraction property} 
\label{su:LocGlobStretch}

While the theory of Section \ref{su:InfinitesimalStretch} gives estimates for
the stretching of tangent vectors via a bound $ \bfL^\mathrm{inf}(E) $ on the
infinitesimal sublevel stretching rate, we now want to obtain contraction
estimates between different solutions. The principal strategy follows the
Eulerian calculus developed in \cite{OttVil00GITL, OttWes05ECCW, DanSav08ECDC}
using the geodesic distance $ \calD^k_N $ induced by the local metric tensor
$\bbG_N$.  For two solutions $ p^N (t)$ and $ \wt{p}^N (t)$, we choose a
geodesic curve $\gamma: [0,1]\ni s \mapsto \gamma(s) \in \bfP_N $ between the
two initial points $ p^N (0)$ and $ \wt{p}^N (0)$. Now we consider the curves
$S^N_t {\circ} \gamma: s \mapsto S^N_t(\gamma(s))$ generated by transporting
the initial geodesic $\gamma$ along the flow $(S^N_t)_{t \geq 0}$ of the
gradient system. Since $S^N_t\circ \gamma$ connects
$ p^N (t) = S^N_t( p^N (0))$ and $ \wt{p}^N (t)=S^N_t( \wt{p}^N (0))$, the
length of $ S^N_t\circ \gamma$ provides an upper bound for the geodesic
distance of $ \calD^k_N( p^N(t), \wt{p}^N(t)) $. Now, the change of the total
length of $ S^N_t\circ \gamma$ can be controlled by the local stretching rates.

However, the application of this argument needs the control of the stretching
rates for all $t$ and all $s$ along the points $S^N_t(\gamma(s))$. This leads
us to the definition of a \emph{local} and a \emph{global sublevel stretching rate}. 

The following definition uses the \emph{local geodesic energy level}
$\GE_N(\Delta, E) \geq E$, where $E > \inf{\calE_N}$ and $\Delta > 0$, which is
defined to be the minimal energy $\widetilde{E}$ such that any
two points $p_0^N, p_1^N \in \Sigma_N(E)$ with
$\calD^k_N(p_0^N, p_1^N) \leq \Delta$ can be connected by a geodesic $\gamma^N$
satisfying $\gamma^N(s) \in \Sigma_N(\wt{E})$ for all $s \in [0, 1]$. More
precisely,
\begin{multline*}
    \GE_N(\Delta, E) := \inf \Bigl\{
      \wt{E} \geq E \Bigm| \forall (p_0^N, p_1^N) \in \Sigma_N(E) \times \Sigma_N(E) 
      \text{ with } \calD^k_N(p_0^N, p_1^N) \leq \Delta, \\
      \exists\; \text{geodesic } \gamma^N \text{ from } p_0^N \text{ to } p_1^N \; 
      \forall s \in [0, 1] : \gamma^N(s) \in \Sigma_N(\wt{E})
    \Bigr\}.
\end{multline*}
Since limits of geodesics are again geodesics and $\calE$ is lower
semicontinuous, one may see that in fact $\GE_N(\Delta, E)$ lies in the
above set, i.e.\ the infimum is attained. 

Clearly, $\Delta \mapsto \GE_N(\Delta, E)$ and $ E  \mapsto \GE_N(\Delta,
E)$ are nondecreasing, and
Lemma~\ref{le:loc.Delta} will show $\GE_N(\Delta, E) \to E$ as
$\Delta \searrow 0$. For the following definition recall that $E \mapsto
\Lambda_N^\mathrm{inf}(E)$ is nonincreasing. 

\begin{definition}[Local and global sublevel stretching rate]
\label{de:GlobSublevSR}
Let $\Delta > 0$ and $E > \inf{\calE_N}$. The \emph{local sublevel stretching
  rate} is defined to be
\begin{equation*}
  \Lambda_N^\mathrm{loc}(\Delta, E) := \Lambda_N^\mathrm{inf}\bigl( 
  \GE_N(\Delta, E)\bigr).
\end{equation*}
The map $\Delta \mapsto \Lambda_N^\mathrm{loc}(\Delta, E)$ is nonincreasing, so
the \emph{global sublevel stretching rate} is defined by
\begin{equation*}
  \Lambda^\mathrm{glo}_N(E) := \lim_{\Delta \to \infty} \Lambda_N^\mathrm{loc}(\Delta, E).
\end{equation*}
\end{definition}

\begin{example}[Local versus global stretching rate]
  \label{ex:LocGlobStretch} We consider the scalar gradient system
  $(\R,\calE,\mafo{id}_\R)$ with $\calE(u)=(u^2{-}1)^2/4$. The Hessian is
  $\bbH(u)=3u^2{-}1$. 

  For $E=1/9$, the sublevel $\Sigma(E)=[-\sqrt{5/3}, - \sqrt{1/3}]\cup
  [\sqrt{5/3}, \sqrt{1/3}] $ consists of two disjoint intervals, and we 
  have $\Lambda^\mafo{inf}(1/9)=0$. Hence, for $\Delta <2/\sqrt{3}$
  two points $u_0,u_1\in \Sigma(1/9)$ with $|u_1{-}u_0|\leq \Delta$ lie in
  the same interval, which implies $\GE(\Delta,1/9) = 1/9$, and 
  $\Lambda^\loc(\Delta,1/9)=0$ follows.  

  For $\Delta \geq 2/\sqrt3$ geodesics may connect points in different
  intervals and pass the point $u=0$ with $\calE(0)=1/4$. Thus, we
  find $\GE(\Delta,1/9)= 1/4$, obtain $ \Lambda^\loc(\Delta,1/9)=-1=\bbH(0)$,
  and arrive at $\Lambda^\glo(1/9) = -1$.
\end{example}

So far we are not able to control the global stretching rate. 
However, the following result shows that it is
possible to control the local sublevel stretching rate in such a way that for
all $E > \inf \calE$ there is a $\Delta > 0$ such that
$\Lambda^\loc_N(\Delta,E)$ can be bounded from below uniformly in $N$ in terms
of $ \bfL^\mathrm{inf}(E{+}1) $.

\begin{lemma}[Lower bound on $ \Lambda_N^\loc(\Delta_E, E)$]
\label{le:loc.Delta}
For all $E > \inf{\calE}$ we have
\begin{equation*}
  \lim_{\Delta \searrow 0} \GE_N(\Delta, E) = E.
\end{equation*}
In particular, for all $E > \inf{\calE}$ there exists $\Delta_E > 0$ such that
\begin{equation*}
  \GE_N(\Delta_E, E) < E + 1 \ \text{ and } \ 
  \Lambda_N^\loc(\Delta_E, E) \geq \bfL^\mathrm{inf}(E {+} 1) \quad 
  \text{ for all }N \in \bbN.
\end{equation*}
\end{lemma}
\begin{proof}
Note that for all $\Delta > 0$ we have $\GE_N(\Delta, E) \leq E^*_N< \infty$,
because $E_N^*:=\sup_{p \in \bfP_N} \calE_N(p) < \infty$ and geodesics exist
between any two $p_0^N, p_1^N \in \bfP_N$. Set $\wt{E} := \GE_N(1, E)$. By
semiconvexity of the energy $\calE_N$ along geodesics, we have that if $\gamma$
is a (constant-speed) geodesic in $(\bfP_N, \calD^k_N)$ wholly contained in the
sublevel $\Sigma_N(\wt{E})$, then
\begin{equation*}
  \calE_N(\gamma(s)) \leq (1 {-} s)\calE_N(\gamma(0)) + s \calE_N(\gamma(1)) 
  - \frac{1}{2} s(1{-} s) \Lambda_N^\mathrm{inf}(\wt{E}) 
   \calD^k_N(\gamma(0), \gamma(1))^2.
\end{equation*}
Thus, assuming $\gamma(0), \gamma(1) \in \Sigma_N(E)$, we find the estimate
\begin{equation*}
  \calE_N(\gamma(s)) \leq E + \frac{1}{8}\max\bigl\{ 0, -\Lambda^\mathrm{inf}_N
    (\wt{E}) \bigr\} \,\calD^k_N(\gamma(0), \gamma(1))^2.
\end{equation*}
It follows that for $\Delta \in (0, 1)$,
\begin{equation*}
  \GE_N(\Delta, E) \leq E + \frac{1}{8} \max\bigl\{ 0, -\Lambda^\mathrm{inf}_N
  (\wt{E}) \bigr\}  \Delta^2 \to E \; \text{ as } \; \Delta \searrow 0,
\end{equation*}
as required. The remaining statements are a consequence of the limit and
Corollary~\ref{co:LowBddLambda}. In particular, we can choose any
$\Delta_E>0$  such that $\frac18 \max\bigl\{ 0, -\Lambda^\mathrm{inf}_N
  (\wt{E}) \bigr\}  \Delta^2_E \leq 1$. 
\end{proof}

\subsection{Semigroup with local sublevel contraction property}
\label{su:SGLocSublContract}

In this section we prove the strong convergence of the solutions
$t \mapsto p^N(t)$ to a limit function $t \mapsto p(t)$ and that the limit
depends continuously on the initial data. The main technical result is the
following.

\begin{proposition}[Comparison of approximations]
\label{pr:ComparApprox}
Let the assumptions \eqref{eq:k(p)bounds} and \eqref{eq:k.W.conds} 
hold and consider $p_0, p_1 \in \Sigma(E) \subset \bfP$.  Then, there exists a
$\Delta>0$ and $N_*\in \N$ such that for all $N, M \geq N_*$ the following
holds. If $\calD^k_{NM}(\bbP_N p_0,\bbP_M p_1)\leq \Delta$, then the unique
solutions $t \mapsto p_0^N(t) \in \bfP_N $ and $t \mapsto p_1^M(t) \in \bfP_M $
of \eqref{eq:DiscrODE} with initial conditions $p_0^N(0)=\bbP_N p_0$ and
$p_1^M(0)= \bbP_M p_1$ satisfy, for all $t>0$, the estimate
\[
\calD^k_{NM} \big( p^N_0(t),p^M_1(t) \big) 
\leq  \ee^{-\bfL^\mathrm{inf}(E+2) \,t} \,\calD^k_{NM} \big( p^N_0(0),p^M_1(0) \big)
      + \mathsf M_{\bfL^\mathrm{inf}(E+2) }(t) \sqrt{\,\ol \kappa\,}\, 
              \big\| G^N - G^M \big\|_{\rmL^\infty},
\]
where $\bfL^\mathrm{inf}(E{+}2)$ is defined in Corollary~\ref{co:LowBddLambda} and
$\mathsf M_\lambda$ is defined in Proposition~\ref{pr:Growth.yN}.
\end{proposition} 
\begin{proof}
By Lemma \ref{le:WellPrepIC} we see that there exists $N_* \in \bbN$ such that 
$\calE_N(\bbP_N p_j)\leq E{+}1$ for $N \geq N_*$ and $j=0,1$. Applying Lemma
\ref{le:loc.Delta} with $E$ replaced by $E{+}1$, we find
$\Delta = \Delta_{E+1}$ such that $\GE_{NM}(\Delta, E{+}1) < E+2$ and
$\Lambda^\loc_{NM}(\Delta,E{+}1)\geq \bfL^\mathrm{inf}(E{+}2)  $.

By construction of $\Delta$ and the assumption 
$\calD^k_{NM}(\bbP_N p_0,\bbP_M p_1)\leq \Delta $ we now find a geodesic curve
$\gamma$ in $\big( \bfP_{NM}, \calD^k_{NM}\big)$ such that 
\[
\gamma(0)=\bbP_N p_0 =p^N_0(0), \quad 
 \gamma(1)=\bbP_M p_1 =p^M_1(0), \quad 
\gamma(s) \in \Sigma_{NM}(E{+}2) \text{ for } s\in [0,1].
\]
For $s \in  [0,1]$ we now consider the gradient-flow solutions associated with 
the gradient system $\big(\bfP_{NM}, \calE^s_{N,M},\bbK_{NM}\big)$ with 
initial data $ % Added "the gradient system" to get rid of overfull hbox.
p^{NM}_s(0)= \gamma(s)$, where  
\[
\calE^s_{N,M}(p^{NM}):= \int_\Omega \Big( W\big(p^{NM}(x)) -\big(
(1{-}s) G^N(x) + s G^M(x)\big) p^{NM}(x) \Big) \dd x.
\]
Thus, the associated gradient-flow equation reads 
\begin{equation}
  \label{eq:ParamODE}
   \dot p_s^{NM} = -
\bbK_{NM}(p_s^{NM}) \big( W'(p_s^{NM}) - (1{-}s) G^N - s G^M \big), \quad 
p^{NM}_s(0)= \gamma(s) . 
\end{equation}
The importance is that for $s=0$ one obtains exactly the solution $p_0^N$
and for $s=1$ the solution $ p_1^M$. 

The parametrized ODE (of dimension $NM{-}1$) together with the initial
conditions depends smoothly on the parameter $s \in [0,1]$. Hence, we can
differentiate the equation with respect to $s\in [0,1]$ and obtain
$y_s^{NM}(t) = \pl_s p^{NM}_s(t)$ such that for each $s\in [0,1]$ the pair
$ \big( p^{NM}_s, y_s^{NM}\big) $ solves \eqref{eq:ODE.Tang}, where $N$ is
replaced by $NM$ and $\zeta = G^M-G^N\in \bfT^*_{NM}$.

Thus, for each $s\in [0,1]$ we can apply  
Proposition \ref{pr:Growth.yN} and use that for $t=0$ we have 
\[
\calD^k_{NM} \big( p^{NM}_0(0),p_1^{NM}(0)\big) 
=\calD^k_{NM} \big( \bbP_N p_0,\bbP_M p_1\big) 
= \int_0^1 \norm{ \pl_s\gamma }{\bbG_{NM} } \dd s  
= \int_0^1 \norm{ y^{NM}_s (0)}{\bbG_{NM}}  \dd s, 
\]
whereas for $t>0$ we have the estimate   
\[
\calD^k_{NM} \big( p_0^N(t),p_1^M(t)\big) 
= \calD^k_{NM} \big( p_0^{NM}(t),p_1^{NM}(t)\big) \leq 
 \int_0^1 \norm{ y^{NM}_s (t)}{\bbG_{NM}}  \dd s
\]
because the smooth curve $s \mapsto p^{NM}_s(t)$ connects 
$ p^{NM}_0(t)= p_0^N(t)$ and $ p^{NM}_1(t)= p_1^M(t)$ in $\bfP_{NM}$. 
Combining this with estimate \eqref{eq:Growth.yN} the desired result
follows. 
\end{proof}

The above result can be used in two ways. First we are able to construct a
unique solution $p(t)=S_t(p_0)$ for all $p_0\in \dom(\calE)= \bigset{p \in
  \bbP}{\calE(p)< \infty } $. Second, we can
conclude that the global semigroup $\big(S_t)_{t\geq 
  0}$ defined on $\dom(\calE)$ has the local sublevel contraction property in
the metric space $(\bfP,\calD^k)$. 

\begin{theorem}[Global semiflow with local contraction]
\label{th:SGLocContract}
Assuming \eqref{eq:k(p)bounds} and \eqref{eq:k.W.conds} hold, 
the global semigroups
$ \big( S^N_t \big)_{t\geq 0} $ on $\dom(\calE_N)$ constructed in Proposition
\ref{pr:ODE.GlobExi} converge to a global semiflow $\big( S_t\big)_{t\geq 0}$
on $\dom(\calE)$ in the sense that
\[
\bfP_N \ni p_0^N \to p_0 \in \bfP  \quad \Longrightarrow \quad 
\forall\, t\geq 0:\  S^N_t(p_0^N) \to S_t(p_0), 
\]
where both convergences are meant strongly in $\rmL^1(\Omega)$. The semigroup
$\big(S_t \big)_{t\geq 0}  $ is a local sublevel contraction in $\big( \bfP,
\calD^k\big)$, namely for all $E \geq \inf\calE$ there exists $\wt\Delta > 0$ 
such that  
\begin{equation}
  \label{eq:SG.LocSublContr}
\begin{aligned}  p_0,p_1\in \Sigma(E)& \text{ with } 
\calD^k( p_0,p_1) \leq \wt\Delta  \quad \Longrightarrow \quad 
\\
& \calD^k\big( S_t(p_0), S_t(p_1)\big) \leq
\ee^{-\bfL^\mathrm{inf}(E+2)\,t} \, \calD^k( p_0,p_1) \ \text{ for all } t\geq 0,
\end{aligned}
\end{equation}
Moreover, the energy is a Lyapunov function, i.e.\ $\calE( S_t(p_0))\leq
\calE(S_r(p_0))$ for all $t\geq r\geq 0$. 
\end{theorem}
\begin{proof}
We fix an energy level $E > \inf \calE$ and apply Proposition
\ref{pr:ComparApprox} first with $p_1=p_0 \in \Sigma(E)$. Because of
$\bbP_N p_0 \to p_0$ in $\rmL^1(\Omega)$ we have $\calD^k_{NM}(\bbP_N p_0 , 
\bbP_M p_0) \leq  \Delta $ for sufficiently large $N,M$, where $\Delta
> 0$ is as in Proposition~\ref{pr:ComparApprox}.  For 
$p^N(t)=S^N_t(\bbP_N p_0)$ we hence obtain
\[
  \Bh(p^N(t), p^M(t)) \leq \frac1{C_\mafo{low}} \calD^k_{NM} \big( p^N(t),
  p^M(t) \big)   \to 0\quad \text{as } N,M\to \infty,
\]
because of $\calD^k_{NM} \big( \bbP_N p_0,\bbP_M p_0)\to 0$ and $\| G^N
{-}G^M  \|_{\rmL^\infty}\to 0$. 
Thus, the limit $ S_t(p_0) := \lim_{N\to \infty} S^N_t(\bbP_N p_0) = \lim_{N\to
  \infty} p^N(t)$ exists in $ \bfP \subset \rmL^1(\Omega)$. 

For the contraction property, set
$ \wt\Delta := C_\mafo{low} \Delta/(2C_\mafo{upp}) $ and choose
$p_0,p_1\in \Sigma(E) $ with $ \calD^k(p_0,p_1) \leq \wt\Delta$. Then, for $N$
sufficiently large, Lemma~\ref{le:CompareDist} and
Proposition~\ref{pr:def.calD.k} imply
$\calD^k_{N^2} \bigl( \bbP_N p_0,\bbP_N p_1 \bigr) \leq \Delta $. For such $N$,
the estimate in Proposition \ref{pr:ComparApprox} (with $M = N$) is valid. 
Moreover, we can fix the initial conditions $\bbP_N p_0$ and $\bbP_N p_1$ and
use the estimate in the metric space $ \big( \bfP_{M!}, \calD^k_{M!}\big) $ for
all $M\geq N$ giving the estimate
\[
\calD^k_{M!} \big( S^N_t(\bbP_N p_0),  S^N_t(\bbP_N p_1) \big) 
\leq \ee^{- \bfL^\mathrm{inf}(E+2) \, t}\, \calD^k_{M!} 
 \big( \bbP_N p_0,  \bbP_N p_1 \big) \quad \text{for all } t\geq 0.
\]
Passing to the limit $M\to \infty$ and then to $N\to \infty$,  we recover the 
desired estimate \eqref{eq:SG.LocSublContr}, where we exploit the definition of
$\calD^k$ in \eqref{eq:Def.calD.k}. 

The remaining properties of $ ( S_t )_{t\geq 0}$ follow in a standard way. 
\end{proof}

In cases where we have a uniform contraction rate $ \bfL^\mathrm{inf}(E) 
= \bfL^\mathrm{inf}_*  $ that is
independent of the energy bound $E$ of the sublevels $\Sigma(E)$, one can
extend the semigroup to the closure $\ol{\dom(\calE)}$ of the domain by uniform
continuity. This is no longer the case in our situation where the the
continuity properties of $(S_t)_{t\geq 0}$ deteriorate for sequences
$ (p_n)_{n \in \bbN} \subset \dom(\calE)$ with $\calE(p_n)\to \infty$.

\subsection{Global Lipschitz estimates on sublevels using \texorpdfstring{$\Bh$}{Bh}} 
\label{su:Lipschitz}

We now want to provide global Lipschitz estimates. Recall that all the above
estimates are contraction estimates which means that we have a constant
``\,1\,'' in front of the exponential $\ee^{-\bfL^\mathrm{inf}(E+2) \,t}$. Of
course, we typically have $ \bfL^\mathrm{inf}(E{+}2) <0$ such that we should
rather speak of an expansion. In contrast we now allow for a constant $C>1$ in
front of the exponential which will still give us Lipschitz estimates.

This constant will arise because we want to switch between the distances
$\calD^k_N$ and the Bhattacharya distance $\Bh$. The reason is that for $\big(
\bfP, \Bh\big)$ we have enough information on the geodesic curves 
to control the energy $\calE$ along these curves. The explicit formula for the
Bhattacharya geodesics and assumption \eqref{eq:W.doubling} for $W$ lead to the
following result. 

\begin{corollary}[Energy bound for geodesics in $(\bfP,\Bh)$]
\label{co:EnergyBhGeod} Assume \eqref{eq:W.doubling} holds. Then, for all
$p_0,p_1\in \Sigma_N(E)$ the unique Bhattacharya geodesic $\gamma^\Bh:[0,1]\to
\bfP_N$ lies completely in $\Sigma_N(\wt E)$ with 
\[
\wt E=  2 C_1 E + C_2 +(2C_1{+}1)\|G\|_{\rmL^\infty} .
\]
\end{corollary}
\begin{proof} Let $s\mapsto \gamma^\Bh(s) \in \bfP_N$ be the Bhattacharya geodesic. 
The condition $\calE_N(\gamma^\Bh(s)) \leq \wt E$ follows by combining assumption
\eqref{eq:W.doubling} with Lemma \ref{le:DensBhatt}.
\end{proof} 

Recalling the definition of the bound $ \bfL^\mathrm{inf}(E) $ for the
infinitesimal sublevel stretching rate in Corollary~\ref{co:LowBddLambda}, we
can now define a \emph{global sublevel Lipschitz rate bound}
\begin{equation}
  \label{eq:GlobStreRate}
  \bfL^\glo(E):=  \bfL^\mathrm{inf}(\wt E) \quad \text{with }\wt E \text{
    from Corollary \ref{co:EnergyBhGeod}}.
\end{equation}
 
With this we are now able to state the global and uniform Lipschitz estimates 
for the discrete solutions $t\mapsto p^N(t)=S^N_t(p^N(0))\in \bfP_N$ as well as
for the global semiflow $\big( S_t\big)_{t\geq 0}$ on $\dom(\calE)\subset
\bfP$. Here we essentially use that our viscosities are of Bhattacharya
type, i.e.\  \eqref{eq:k(p)bounds} holds, which is seen in the result where the
constant $C_\mafo{upp}/C_\mafo{low} = \sqrt{\ol\kappa/\ul\kappa}$ appears. Of
course, for the true Bhattacharya case $k(p)=\kappa\,p$, we have
$C_\mafo{upp}/C_\mafo{low} =1$, and the Lipschitz estimates
\eqref{eq:pj.Lipschitz} are indeed contraction/expansion estimates. 

\begin{theorem}[Uniform Lipschitz estimate]
\label{th:UnifLipschitz}
Assume \eqref{eq:k(p)bounds}, \eqref{eq:k.W.conds}, and \eqref{eq:W.doubling}
hold. For $N\in \N$ and $E > \inf \calE_N$ consider
$p_0^N,p^N_1 \in \Sigma_N(E)$. Then, with $\bfL^\glo (E)$ from
\eqref{eq:GlobStreRate} and $C_\mafo{low}$ and $C_\mafo{upp}$ from Lemma
\ref{le:CompareDist}, we have
\begin{subequations}
  \label{eq:pj.Lipschitz}
\begin{equation}
  \label{eq:pj.Lipschitz.a}
  \Bh\big(S^N_t(p_0), S^N_t(p_1) \big) \leq \frac{C_\mafo{upp}}{C_\mafo{low}} \,
  \ee^{- \bfL^\glo(E)\,t}\,  \Bh\big(p^N_0, p^N_1 \big)
  \quad \text{for all } t>0. 
\end{equation}
Similarly, for all $p_0,p_1 \in \Sigma(E)$ we have 
\begin{equation}
  \label{eq:pj.Lipschitz.b}
 { \Bh\big(S_t(p_0), S_t(p_1) \big) \leq \frac{C_\mafo{upp}}{C_\mafo{low}} \,
  \ee^{- \bfL^\glo(E)\,t}\,  \Bh\big(p_0, p_1 \big)
  \quad \text{for all } t>0. }
\end{equation}
\end{subequations}
\end{theorem}
\begin{proof} The proof follows the idea of the proof of Proposition
  \ref{pr:ComparApprox}, but now we do not use a geodesic in the space
  $(\bfP_N,\calD^k_N)$ but in the Bhattacharya space $( \bfP_N,\Bh)$. 
  
\emph{Step 1. Construction of transported geodesic:} By Corollary
\ref{co:EnergyBhGeod} we know that the $\Bh$ geodesic
$\gamma^\Bh\in \rmC^1([0,1];\bfP_N)$ that connects $p^N_0$ and $p^N_1$ 
lies in the sublevel $\Sigma_N(\wt E)$. Now,
we define the transported curves $\gamma_t\in \rmC^1([0,1];\bfP_N)$ via
\[
\gamma_t(s) = S_t^N(\gamma^\Bh(s)) \quad \text{for all } s\in [0,1] \text{ and
} t\geq 0, 
\]
where $(S^N_t)_{t\geq 0}$ denotes the gradient-flow semigroup for
$(\bfP_N,\calE_n,\bbK_N)$. As $\calE_N$ is a Lyapunov function we conclude
\[
   \gamma_t(s) \in \Sigma_N(\wt E) 
 \quad \text{for } s\in [0,1] \text{ and } t\geq 0 .\medskip
\]

\emph{Step 2. Estimating the squared length of $\gamma_t$:} 
In $ (\bfP_N, \calD^k_N) $ we have 
\[
\ell(t):=\mafo{Length}_{\calD^k_N}(\gamma_t)^2 = \int_0^1 \langle \bbG_N(\gamma_t(s))
\pl_s\gamma_t(s),\pl_s \gamma_t(s)\rangle  \dd s \geq
\calD^k_N\big(S^N_t(p_0),S^N_t(p_1)\big)^2
\]
for all $t\geq 0$, where the last estimate holds since
$\gamma_t(j)=S^N_t(p_j)$ for $j = 0, 1$, i.e.\ the 
curves $\gamma_t$ connect $S^N_t(p_0)$ and $S^N_t(p_1)$.  

First, we use $k(p)\geq \ul\kappa \,p$ to obtain
\[
\ell(0) \leq \frac4{\ul\kappa} \mafo{Length}_{\Bh}(\gamma^0)^2 =
C_\mafo{upp}^2 \Bh( p_0, p_1)^2,
\] 
where the last equality holds because $\gamma^\Bh=\gamma^0$ is a $\Bh$-geodesic. 

Secondly, we estimate the growth of $\ell$ as follows. For each fixed
$s\in {]0,1[}$ we can use that the pair
$ t \mapsto (p^N(t),y^N(t))=(\gamma_t(s), \pl_s \gamma_t(s))$ solves
\eqref{eq:ODE.Tang} with $\zeta^N=0$. 
Hence, taking the derivative of $\ell$ with respect to $t>0$ and employing
Theorem \ref{th:Stretching} gives
\begin{align*}
  \frac\rmd{\rmd t} \ell(t)
  &= \int_{s=0}^1 
    2\frac{\pl}{\pl t}\calR_N(\gamma_t(s),\pl_s\gamma_t(s)) 
  \dd s 
% \\ &
= 2 \int_{s=0}^1  {-} \bigl\langle \wt\bbH_N(\gamma_t(s))\pl_s\gamma_t(s) , 
   \pl_s\gamma_t(s) \bigr\rangle   \dd s 
\\
  &\geq -2 \int_{s=0}^1 
    \Lambda^\mafo{inf}_N( \wt E) 2\calR_N(\gamma_t(s),\pl_s\gamma_t(s)) 
  \dd s 
% \\    &
 = \ -2  \Lambda^\mafo{inf}_N(\wt E) \ell(t) 
% \\   & 
 \leq -2 \bfL^\glo(E) \ell(t). 
\end{align*}
Hence, Gr\"onwall's lemma yields $\ell(t) \leq \ee^{ -2\bfL^\glo(E) \,
  t} \ell(0)$.\medskip 

\emph{Step 3. The Lipschitz estimate:} As $\gamma_t(j)=S^N_t(p_j)$ for $j = 0,
1$ and all $t>0$, we can now estimate the distance between the two solutions:
\begin{alignat*}{4}
\Bh\big(S^N_t(p_0),S^N_t(p_1) \big)^2 
  &\leq \frac1{C^2_\mafo{low}} \, \calD^k_N \big(S^N_t(p_0), 
    S^N_t(p_1)\big)^2
  &&\leq \ \frac1{C^2_\mafo{low}}\,\ell(t) 
\\
  &\leq  \frac1{C^2_\mafo{low}} \,\ee^{-2\bfL^\glo(E) \,t}\,\ell(0)\
  &&\leq \  \frac{C_\mafo{upp}^2}{C^2_\mafo{low}} 
    \, \ee^{-2\bfL^\glo(E) \, t}
    \,\Bh\big(p^N_0,p^N_1\big)^2.  \quad\mbox{}
\end{alignat*}
This is the desired result \eqref{eq:pj.Lipschitz.a}, and
\eqref{eq:pj.Lipschitz.b} follows in exactly the same way.
\end{proof}

\section{Characterizations of the semiflow}
\label{se:CharSemiflows} 

In this section we discuss in what sense the constructed global semigroup
$\big( S_t\big)_{t\geq 0} $ on
$\dom(\calE)= \bigset{p\in \bfP}{ \calE(p) <\infty}$ can be understood as a
\emph{gradient flow}, i.e.\ associated to the metric gradient system
$ \big( \bfP, \calE, \calD^k\big)$.  In Section \ref{su:WeakSolutions} we show
that the curves $p(t)=S_t(p(0))$ are weak solutions in the sense that
\eqref{eq:I.BasicEq} holds after testing against smooth test functions. In
Section \ref{su:CurvMaxSlope} we show that the solutions are curves of maximal
slope in the sense defined in \cite{AmGiSa05GFMS}, which follows by
establishing a suitable chain rule. An important consequence is that the energy
converges along the approximations, namely $\calE(p^N(t)) \to \calE(p(t))$
which is nontrivial because of the missing compactness. A similar result was
obtained in \cite{MiOrSe14ANVM}, but only for the Bhattacharya case
$k(p)=\kappa p$. In the final Section \ref{su:SublevelEVI} we show that the
solutions form a sublevel EVI flow when the global sublevel stretching rate
$\Lambda^\glo_N(E)$ is bounded below uniformly in the approximation dimension
$N$. This assumption is a major restriction but it certainly holds in the
Bhattacharya case, see Corollary \ref{co:Bhatt.EVI}.

\subsection{Weak solutions} 
\label{su:WeakSolutions}

Having the uniform convergence of $p^N(t)=S^N_t(p^N(0))$ to $p(t)= S_t(p(0))$
we may also pass to the limit in the following weak form of equation
\eqref{eq:I.BasicEq}. A function $p: {[0, \infty[} \to \bfP$ is called a \emph{weak
solution} if  
\begin{subequations}
\label{eq:WeakSol}
\begin{align} 
  \label{eq:WeakSol.a} & p \in \rmC^0({[0,\infty[},\bfP) \text{ and } 
   \sqrt{k(p)}\,\big( W'(p){-}G \big) \in \rmL^2_\loc\big({[0,\infty[}\ti\Omega
     \big), 
 \\
 \nonumber 
 &\forall \: \varphi\in \rmC^1_\rmc( {[0,\infty[} \ti \ol\Omega): 
 \\ 
 & \label{eq:WeakSol.b} 
  \quad \int_\Omega p(0)\varphi(0)\dd x = \int_0^\infty \hspace{-0.5em} 
   \int_\Omega \biggl( p\,\pl_t \varphi - k_p\biggl( W'(p) {-}G 
  -\frac{\big[k_p(W'(p){-}G)\big]}{[k_p]}\biggr) \varphi \biggr) \dd x \dd t, 
\end{align}
\end{subequations}
where as above $k_p(t,x)=k(p(t,x))$. Note that the second condition in
\eqref{eq:WeakSol.a} implies that the 
integral in  \eqref{eq:WeakSol.b}  exists, because of
$\int_0^T\hspace{-0.4em}\int_\Omega  k_p\Bigl(
W'(p){-}G-\frac{[k_p(W'(p)-G)]}{[k_p]}\Bigr)^2\dd x \dd t 
=   \int_0^T \hspace{-0.4em} \int_\Omega  k_p\big( W'(p)-G)^2 \dd x \dd t - 
\int_0^T \frac{[k_p (W'(p) - G)]^2}{[k_p]} \dd t < \infty $  and $\sqrt{k_p} \in
\rmL^2([0,T]\ti\Omega)$. 

A major role in our proofs will be played by the functions 
\begin{equation}
  \label{eq:psi.psiN}
  \begin{aligned}
\psi(t,x)&:=  \sqrt{k(p(t,x))} \biggl(
  W'(p(t,x)) -G(x) - \frac{[k_{p(t)}(W'(p(t)){-}G)]}{[k_{p(t)}]}
  \biggr) \ \text{ and}
\\
\psi^N(t,x)&:=  \sqrt{k(p^N(t,x))} \biggl(
  W'(p^N(t,x)) -G(x) - \frac{[k_{p^N(t)}(W'(p^N(t)){-}G)]}{[k_{p^N(t)}]}
  \biggr),
\end{aligned}
\end{equation}
because $\psi^N$ has a priori bounds in $\rmL^2({[0,\infty[}\ti \Omega)$ independent
of $N$. 

\begin{theorem}[Weak solutions]
\label{th:weakSol}
Let $p_0 \in \dom(\calE)$, and define $p : [0, \infty[ \to \bfP$ by
$p(t) := S_t(p_0)$. Then $p$ is a weak solution of \eqref{eq:I.BasicEq}.
\end{theorem}
\begin{proof}
By Lemma \ref{le:WellPrepIC} we can choose $p_0^N$ such that
$p_0^N\to p_0$ in $\rmL^1(\Omega)$ and $\calE_N(p^N_0) \to \calE(p_0)$ as
$\N \to \infty$. Setting $p^N(t)= S^N_t(p^N_0)$ Theorem \ref{th:SGLocContract}
shows $p(t) = \lim_{N \to \infty} p^N(t)$ for all $t \in [0, \infty[$,
where the latter limit is taken in $(\bfP, \Bh)$. Note it follows straight from
the definition of the Bhattacharya distance that $p^N \to p$ pointwise a.e.\ in
$[0, \infty[ \ti \Omega$.

\emph{Step 1.} Our strategy starts from the energy-dissipation balance formula
\begin{equation} \label{eq:energydissipation}
  \begin{aligned}
  &  \calE_N(p^N(0)) = \calE_N(p^N(T)) + \int_0^T \bigl\langle
    \rmD\calE_N(p^N(t)), \bbK(p^N(t))\rmD\calE_N(p^N(t)) \bigr\rangle
    \, \rmd{t}, \\
    &= \calE_N(p^N(T)) + \int_0^T \hspace{-0.5em} \int_\Omega \big|
    \psi^N(t,x)\big|^2  \, \rmd{x} \, \rmd{t} 
 \geq \inf{\calE} + \int_0^T \hspace{-0.5em} \int_\Omega \big|
    \psi^N(t,x)\big|^2 \, \rmd{x} \, \rmd{t},
  \end{aligned}
\end{equation}
which holds for all $N \in \bbN$ and $T > 0$. Using the property that
$\calE_N(p^N(0)) \to \calE(p(0))$, it follows that the functions
$\psi^N : [0, \infty[ \ti \Omega \to \bbR$ are bounded in
$\rmL^2({[0,\infty[}\ti \Omega)$ uniformly in $N \in \bbN$.

We claim $\psi^N \rightharpoonup \psi$ in $\rmL^2([0, T] \ti \Omega)$.
This can be seen by noting that $\psi^N \to \psi$ pointwise a.e.\ in
$[0, T] \ti \Omega$, so convergence of
$\int_0^T \hspace{-0.4em} \int_\Omega \psi^N v \, \rmd{x} \, \rmd{t}$ to
$\int_0^T \hspace{-0.4em} \int_\Omega \psi v \, \rmd{x} \, \rmd{t}$ for arbitrary
$v \in \rmL^2([0,T] \ti \Omega)$ follows by H\"{o}lder's inequality and Lebesgue's
dominated convergence theorem.

\emph{Step 2.} Next, we claim $\sqrt{k(p^N)} \to \sqrt{k(p)}$ strongly in
$\rmL^2([0,T] \ti \Omega)$. To see this, we have the elementary estimate
\begin{equation*}
  \bigl\lvert \sqrt{k(p)} - \sqrt{k(p^N)} \bigr\rvert^2
  \leq 2 \bigl( k(p) + k(p^N) \bigr)
  \leq 2\ol{\kappa}(p + p^N)
  \in \rmL^1([0, T] \ti \Omega),
\end{equation*}
so the desired convergence holds again by the Lebesgue's DCT.

\emph{Step 3.} We now show the weak formulation \eqref{eq:WeakSol.b} is
preserved in the limit. Let
$\varphi \in \rmC_\rmc^1([0, \infty[ \ti \ol{\Omega})$ and choose $T>1$
such that $\supp{\varphi} \subset [0,T{-}1]\ti \ol\Omega$. Then,
\begin{equation*}
  \int_\Omega p^N(0)\varphi(0) \, \rmd{x}
  = \int_0^\infty \hspace{-0.5em} \int_\Omega \Bigl(
  p^N \pl_t{\varphi} - \sqrt{k(p^N)}\: \psi^N \varphi
  \Bigr) \, \rmd{x} \, \rmd{t}.
\end{equation*} 
The convergence of the first and second term follow by linearity and
the convergence of $p^N\to p \in \rmC^0\big([0,T];\rmL^1(\Omega)\big)$.  By
Steps 1 we have $ \psi^N \rightharpoonup \psi $ (weakly), and Step 2 yields
$\sqrt{k(p^N)} \varphi \rightarrow \sqrt{k(p)} \varphi $ in
$\rmL^2([0,T]\ti \Omega)$ (strongly). Hence the third term converges to the
desired limit as well, and we arrive at
  \begin{equation*}
    \int_\Omega p(0)\varphi(0) \, \rmd{x}
    = \int_0^\infty \hspace{-0.5em} \int_\Omega \Bigl(
      p\, \pl_t{\varphi} - \sqrt{k(p)} \:\psi\, \varphi
    \Bigr) \, \rmd{x} \, \rmd{t}.
  \end{equation*}
  This is precisely \eqref{eq:WeakSol.b}.
\end{proof}

\subsection{Curve of maximal slope}
\label{su:CurvMaxSlope} 

Here we provide conditions under which we are able to show that the solutions
$t \mapsto p(t) =S_t(p_0)$ with $ p_0 \in \dom(\calE)$ are curves of maximal
slope for the gradient system $\big( \bfP,\calE, \calR \big)$ with 
dissipation potential $\calR$ from \eqref{eq:def.calR}. 

\begin{remark}[Metric versus Riemannian gradient systems]\slshape
\label{re:MetricRiemGS}
Note that we do not use the theory of metric gradient flows from
\cite{AmGiSa05GFMS} which relies in the metric speed $\norm{ \dot p}{\calD^k} $
and the the metric slope $\norm{ \pl \calE}{\calD^k}$. Instead we use the
explicit dissipation potentials $\calR(p,\dot p)$ and $\calR^*(p,-\rmD
\calE(p))$. To connect our approach to the metric theory, it would be necessary
to show  the characterizations  
\[
\norm{ \dot p}{\calD^k}(t)= \sqrt{2 \calR(p(t),\dot p(t))} \quad
\text{and} \quad \norm{ \pl \calE}{\calD^k}(p) =
\sqrt{2\calR^*(p,-\rmD\calE(p)) }.
\]
This is trivial in the finite-dimensional case $\bfP_N$
(cf. \cite{Miel23IAGS}) but more tricky in $\bfP$. 
\end{remark}

The proof follows the usual steps of showing that (i) the limits
$p(t)=S_t(p_0)$ are EDI solutions (EDI = Energy-Dissipation Inequality) and
(ii) then showing that a suitable chain rule holds. This then implies that the
solutions satisfy the energy-dissipation balance (EDB), which yields the fact
that we have curves of maximal slope and that the energy of the approximations
converges to the energy of the limit solutions.

\begin{lemma}[EDI solutions]
\label{le:EDIsol} 
The solutions $p(t)=S_t(p_0) \in \bfP$ with $p_0\in \dom(\calE)$ are EDI
solutions, i.e.\ they satisfy, for all $T>0$, the estimate
\begin{equation}
  \label{eq:EDIsol}
  \calE(p(T)) + \int_0^T\!\!\int_\Omega  \biggl(\frac{(\dot p)^2}{2k(p)} +
\frac{k(p)}{2} \biggl(W'(p){-}G - \frac{[k_p(W'_p{-}G)]}{[k_p]} 
  \biggr)^{\!\!2\ } \biggr) \dd
t \leq \calE(p_0) .
\end{equation}
\end{lemma}
\begin{proof} We fix an arbitrary $T>0$ and start from the approximations
$p^N(t)=S^N_t( \bbP_N p_0)$ as defined in Theorem \ref{th:SGLocContract}. 

These ODE solutions are EDB solutions, namely
\begin{equation}
  \label{eq:EDB.pN} 
  \calE(p^N(T)) + \int_0^T\!\!\int_\Omega  \biggl(\frac{(\dot p^N)^2}{2k(p^N)} +
  \frac12 \big( \psi^N\big)^2\biggr) \dd t = \calE( \bbP_N p_0).
\end{equation}
with $ \psi^N $ from \eqref{eq:psi.psiN}. 
In particular, $p^N \in \rmA\rmC_2\big([0,T];(\bfP,\Bh)\big)$ with a bound
independent of $N$, namely
\begin{align*}
\int_0^T \!\! \int_\Omega \frac{(\dot p^N)^2}{k(p^N)} \dd x\dd t 
\leq 2\big(\calE(\bbP_N p_0) -\inf \calE\big) \leq 2\big(\calE(p_0)+1 -\inf
\calE\big).
\end{align*}
As $p^N(t,x)\geq \delta_N>0$ (see Lemma \ref{le:SublevPositive}), we can define
$\eta^N:= \dot p^N/\sqrt{k(p^N)}$ and find that $\eta^N$ is uniformly bounded
in $\rmL^2([0,T]\ti \Omega)$. Hence, there is a subsequence such that
$\eta^{N_k} \rightharpoonup \eta^\infty$ in $\rmL^2([0,T]\ti \Omega)$. Passing
to the limit in the weak formulation of $\dot p^{N_k} = \sqrt{k(p^{N_k})} 
\, \eta^{N_k}$ as in the proof of Theorem
\ref{th:weakSol}, we find that the equation
$\dot p = \sqrt{k(p)} \,\eta^\infty$ is also satisfied in a weak sense. 
Since $k(p)>0$ a.e.\ we conclude that $\eta^\infty$ is uniquely determined,
which implies $\eta^{N} \rightharpoonup \eta^\infty$ in
$\rmL^2([0,T]\ti \Omega)$ along the whole sequence, not just along the
subsequence $(\eta^{N_k})_{k \in \bbN}$. Thus, passing to the limit
$N \to \infty$ via Fatou's lemma we conclude 
\begin{align*}
V_T&= \int_0^T\!\!\int_\Omega \frac{(\dot p)^2}{k(p) }\dd x \dd t
 = \int_0^T\!\!\int_\Omega \big(\eta^\infty \big)^2 \dd x \dd t 
 \leq \liminf_{N\to \infty} V^N_T
\\
&\text{with } V^N_T:=\int_0^T\!\!\int_\Omega  
           \big(\eta^N\big)^2 \dd x \dd t
= \int_0^T\!\!\int_\Omega 
    \frac{(\dot p^N)^2}{k(p^N)} \dd x \dd t.
\end{align*}
  
Similarly, we can pass to the limit in the other terms of \eqref{eq:EDB.pN},
namely using the lemma of Fatou for the first and the third term, exploiting
the convergence $p^N(t)= S^N_t(\bbP_N p_0) \to
p(t)=S_t(p_0)$ for all $t>0$ guaranteed by Theorem \ref{th:SGLocContract}: 
$\calE(p(T))\leq \liminf_{N\to \infty} \calE(p^N(T))$ and 
\[
U_T:=  \int_0^T\!\!\int_\Omega \psi^2 \dd x \dd t 
 \leq \liminf_{N\to \infty} U^N_T  \quad \text{with }
 U^N_T:=\int_0^T\!\!\int_\Omega \big(\psi^N\big)^2 \dd x \dd t ,
\]
with $\psi$ from \eqref{eq:psi.psiN}. The fourth and last term
satisfies $\calE(\bbP_N p_0) \to \calE(p_0)$ as is shown in Lemma
\ref{le:WellPrepIC}. This finishes the proof.
\end{proof}

The next step consists in showing the existence of a chain rule for all
functions satisfying the assumptions given by the formulation of the EDI
solutions. In the theory of metric gradient systems this corresponds to showing
that the metric slope is a strong upper gradient. 

\begin{lemma}[Chain rule] 
\label{le:CahinRule} 
Let the assumptions \eqref{eq:k(p)bounds}, \eqref{eq:k.W.conds}, and 
\begin{equation}
  \label{eq:Monotone}
  \exists \, p_*>1 \ \forall \, p\in {]0,1/p_*]} \cup {[p_*,\infty[}: \ \ 
(p{-}1)W'(p) \geq 0 \text{ and } k(p)W''(p)+\tfrac{k'(p)}2 W'(p) \geq 0
\end{equation}
hold. Consider $T>0$ and a function $p :[0,T]\to \bfP$ such that 
\[
\sup_{t\in [0,T]} \calE(p(t)) < \infty \text{ and }
 \int_0^T\!\!\int_\Omega  \biggl( \frac{(\dot p)^2}{2k(p)} +
  \frac{k(p)}{2} \biggl( W'(p){-}G - \frac{[k_p(W'_p{-}G)]}{[k_p ]} 
     \biggr)^2 \biggr) \dd t  < \infty.
\]
Then, for all $t_0,t_1\geq 0$ with $t_0 < t_1$ we have 
\begin{equation}
  \label{eq:ChainRule}
  \calE(p(t_1)) - \calE(p(t_0)) = \int_{t_0}^{t_1} \!\!\int_\Omega \dot p \:\biggl(
W'(p){-}G - \frac{[k_p(W'_p{-}G)]}{[k_p]} \biggr) \dd x \dd t. 
\end{equation}
\end{lemma}
\begin{proof} The result is established by approximation. 

First, we assume that
there exists an $R>1$ such that $p$ satisfies the lower and upper bound
$p(t,x)\in [1/R,R]$ a.e. Then we can use that $k$, $W$, and $W'$ are uniformly
continuous on $[1/R,R]$ and that $p\in \rmH^1([0,T];\rmL^2(\Omega))$. 
Thus, \eqref{eq:ChainRule} follows from standard arguments. 

Secondly, we consider a general function $p$ satisfying the assumptions. Moreover,
we fix $T>0$ and assume without loss of generality $t_0=0$ and $t_1=T$. 
For all $R>1$ we define the approximations $p_R$ and the auxiliary functions
$\eta_R$ and $\psi_R$ via 
\begin{align*}
p_R &:= \max\big\{  1/R,\ \min\{ p, R\}\,\big\} \ \in [1/R, R], 
\\ 
\eta_R &:= \frac{\dot p_R}{\sqrt{k(p_R)}},\quad  
\psi_R:= \sqrt{k(p_R)}\Big(W'(p_R){-}G -
       \frac{[k_{p_R}(W'_{p_R}{-}G)]}{[k_{p_R}]}  \Big)  ,
\end{align*}
Observe that $p_R$ does not necessarily lie in $\bfP$, but that is not needed
here. As shown above, the chain rule \eqref{eq:ChainRule} holds for $p_R$, and
it can be rewritten as 
\[
 \calE(p_R(T)) - \calE(p_R(0)) = \int_0^T \!\!\int_\Omega \eta_R \;
 \psi_R  \dd x \dd t. 
\]

The strategy is to pass to the limit $R\to \infty$. By the first condition in
\eqref{eq:Monotone} we have that $R \mapsto W(p_R(t,x))$ is monotone in
$R\in {[p_*,\infty[}$. Hence, Beppo Levi's monotone convergence theorem implies
$\calE(p_R(t)) \to \calE(p(t))$ for all $t\in [0,T]$.

Similarly we observe that $R\mapsto \eta_R =\dot p_R/\sqrt{k(p_R)}$ is
monotone, because it equals $\dot p/\sqrt{k(p)}$ a.e.\ where
  $|p(t,x)|\leq R$ and $0$ otherwise. Thus we have $|\eta_R|\leq |\eta|$
  and $\|\eta_R\|_{\rmL^2}^2 \to \| \eta\|_{\rmL^2}^2$ by monotone
  convergence. This implies the strong convergence 
\[
\eta_R \to \eta := \frac{\dot p}{\sqrt{k(p)}} \quad \text{in }
  \rmL^2([0,T]\ti \Omega ).
\]

For showing convergence of $\psi_R$, we derive additional a priori
bounds for the function $p$. Let $\mathsf{C_E}:=\sup\calE(p(t)) < \infty$. 
By \eqref{eq:k(p)bounds} and the third condition in \eqref{eq:Cond.lambda.cvx}
we find  
\[
\bigl| \,\bigl[ k_p(W'_p{-}G)\bigr]\,\bigr| 
\leq B_1\calE(p(t))  + B_2 + \ol\kappa \|G\|_{\rmL^\infty} 
\leq B_1\mathsf{C_E} + B_2 +
\ol\kappa \|G\|_{\rmL^\infty}=:C_1 < \infty.
\] 
The same bound also holds with $p$ replaced by $p_R$ because $\calE(p_R)\leq
\calE(p)\leq \mathsf{C_E}$. Similarly, we can bound $ [k_{p_R}G W'_{p_R}]$, and
$[k_{p_R}G^2]$ uniformly in $R\geq p_*$ and $t\in [0,T]$.

Using also $[k_p] \in [\ul\kappa, \ol \kappa]$ we see that the assumption 
$\int_0^T \!\! \int_\Omega \psi^2 \dd x \dd t < \infty$ implies the bound 
$ \int_0^T\!\!\int_\Omega k(p) \big( W'(p)\big)^2 \dd x \dd t < \infty$. 
We now exploit condition \eqref{eq:Monotone} once again. For $\Phi(p)
:=k(p)\big(W'(p)\big)^2$ we find $\Phi'(p)=W'(p)\big( 2 k(p)W''(p)+
k'(p)W'(p)\big)$, hence $p\mapsto \Phi(p)$ is decreasing for $p\in {]0,1/p_*]}$
and increasing for $p\in {[p_*,\infty[}$. Arguing as above we conclude 
\[
\zeta_R:=\sqrt{k(p_R)}\, W'(p_R) \ \to   \ \zeta:= \sqrt{k(p)}\, W'(p) 
 \quad \text{strongly in }  \rmL^2([0,T]\ti \Omega).
\]
The representation $\psi_R=\zeta_R+ \sqrt{k(p_R)}\big( -G +
[k_{p_R}(W'_{p_R}{-}G)]/[k_{p_R}]\big)$ and the above a priori estimates allow
us to apply Lebesgue's DCT to $\psi_R-\zeta_R$ and we find the desired
convergence $\psi_R\to \psi $ in $\rmL^2([0,T]\ti \Omega)$.
  
With this and the fact that $\eta \psi$ is exactly the desired integrand,
the chain rule \eqref{eq:ChainRule} follows for general $p$. 
\end{proof}

After these preparatory results we are ready to state out main result
concerning the property of curves of maximal slope. As emphasized in Remark
\ref{re:MetricRiemGS} we do not use the metric frame work but 
consider the Riemannian gradient system $\big( \bfP,\calE, \calR\big)$ using 
\begin{align*}
&\calR(p,\dot p)=\frac12\langle \bbG(p)\dot p, \dot p\rangle = \int_\Omega
\frac{(\dot p)^2}{2\:k(p)}  
 \dd x \ \text{ and } 
\\ 
&\calR^*\big(p,{-}\rmD\calE(p)\big)=\frac12\langle \rmD\calE(p),
\bbK(p)\rmD\calE\rangle =  
\int_\Omega \frac{k(p)}2 \Big(W'(p){-}G - \frac{[k_p(W'_p{-}G)]}{[k_p]} \Big)^2
\dd x .
\end{align*}

\begin{theorem}[Curves of maximal slope]
\label{th:CurvesMaxSlope} 
Assuming conditions \eqref{eq:k(p)bounds}, \eqref{eq:k.W.conds}, and \eqref{eq:Monotone}, the solutions $p(t)=S_t(p_0)$ with $p_0 \in \dom(\calE)$ are curves of maximal
slope, in the sense that $t \mapsto \calE(p(t))$ is absolutely continuous and
\begin{equation}
  \label{eq:Evol.CMS}
\frac\rmd{\rmd t} \calE(p(t)) = -\calR \big( p(t) \dot p(t)\big) - 
 \calR^* \big( p(t), {-} \rmD\calE(p(t)) \big) \quad \text{a.e.\ in } {]0,\infty[}.
\end{equation}
In particular we have the equation 
\begin{equation}
  \label{eq:Evol.L2}
  \dot p = - \sqrt{k(p)}\, \psi \text{ with } \psi =  \sqrt{k(p)} \Big( W'(p){-}G -
\frac{[k_p(W'_p{-}G)]}{[k_p]}\Big) \in \rmL^2([0,T]\ti \Omega). 
\end{equation}
Because of $\sqrt{k(p)}\in \rmL^\infty([0,T];\rmL^2(\Omega)$ we further have
$\dot p \in  \rmL^1([0,T]\ti \Omega)$. 

Moreover, the approximations $p^N(t)=S^N_t(\bbP_N p_0)$ satisfy
$\calE(p^N(t))\to \calE(p(t))$ for all $t>0$, i.e.\ well-preparedness is
preserved for all positive times.  
\end{theorem}
\begin{proof} The proof is now straight forward. We let $\eta =\dot
  p/\sqrt{k(p)}$ such that  
\[
\big\langle \bbG(p(t))\dot p(t), \dot p(t)\big\rangle =
\big\|\eta\big\|_{\rmL^2(\Omega)}^2 \quad \text{and}\quad 
\big\langle \rmD\calE(p(t)),  
\bbK(p(t))\rmD\calE(p(t)) \big\rangle= \big\|\psi\big\|_{\rmL^2(\Omega)}^2 .
\]

Lemma \ref{le:EDIsol} shows 
$\calE(p(T))+ \frac12\int_0^T(\|\eta\|_{\rmL^2}^2 {+}\|\psi\|_{\rmL^2}^2 )\dd
t\leq \calE(p(0))$. In particular, the assumption of the Lemma
\ref{le:CahinRule} holds and we can insert the chain rule to eliminate the
energy difference $\calE(p(T))-\calE(p(0))$, which leads to
\begin{equation}
 \label{eq:EDBviaCR}
  \begin{aligned}
   0 
   \overset{\text{EDI}}&{\geq} \int_0^T\! \Big( \frac12 \|\eta\|_{\rmL^2}^2 
   {+}\frac12\|\psi\|_{\rmL^2}^2 \Big) \dd t + \calE(p(T))-\calE(p(0)) 
  \\
   \overset{\text{(chain rule)}}&{=} \int_0^T\!\!\int_\Omega \Big( \frac12\eta^2
      + \frac12 \psi^2 + \eta \,\psi\Big) \dd x \dd t
    = \frac12 \int_0^T\!\!\int_\Omega \big( \eta +\psi\big)^2  \dd x \dd t 
    \geq 0.
\end{aligned} 
\end{equation}
Thus, we conclude $\eta=-\psi$ and \eqref{eq:Evol.L2} is established. 

Moreover, using $\eta\psi\in \rmL^1([0,T]\ti \Omega)$ and the chain rule
\eqref{eq:ChainRule} we conclude that  $t\mapsto \calE(p(t))$ is absolutely
continuous with 
\[
\frac\rmd{\rmd t} \calE(p(t)) = \int_\Omega \eta(t)\psi(t)\dd x = -
\int_\Omega \Big( \frac12\,\eta^2 +\frac12\,\psi^2\Big) \dd x, 
\]
and \eqref{eq:Evol.CMS} is established as well. 

For the final statement we go back to the proof of Lemma~\ref{le:EDIsol}
and use the approximate solutions $p^N(t)=S^N_t(\bbP_N p(0))$ which define
$\eta^N$, $\psi^N$, $V_T^N = \| \eta^N\|_{\rmL^2}^2$, and $ U_T^N = \|
\psi^N\|_{\rmL^2}^2$. There we have shown 
\begin{subequations}
\label{eq:Liminf.Est} 
\begin{align}
&\calE(p(T))\leq  \liminf_{N\to \infty} \calE(p^N(T)), \quad 
V_T\leq \liminf_{N\to \infty} V^N_T, 
\quad \text{and } U_T\leq \liminf_{N\to \infty} U^N_T ,
\\
&\begin{aligned}
  &\calE(p(T))+ \tfrac12 V_T+\tfrac12 U_T
   \leq \liminf_{N\to \infty} \Big( \calE(p^N(T)) +  \tfrac12 V^N_T +  
    \tfrac12 U^N_T\Big)\\ &\quad  \leq \limsup_{N\to \infty} \Big( \calE(p^N(T))
    + \tfrac12 V^N_T + \tfrac12 U^N_T\Big) 
   = \lim_{N\to \infty} \calE(\bbP_N p(0))  = \calE(p(0)).
\end{aligned}
\end{align}
\end{subequations}
However, \eqref{eq:EDBviaCR} gives
$\calE(p(0))= \calE(p(T))+ \tfrac12V_T+\tfrac12U_T$ such that all inequalities
in \eqref{eq:Liminf.Est} must be equalities. In particular, 
$\calE(p^N(T))\to \calE(p(T))$, $V^N_T\to V_T$, and $U^N_T\to
U_T$, and the result follows as
the choice of $T>0$ is arbitrary.
\end{proof}

\subsection{Sublevel EVI solutions}
\label{su:SublevelEVI} 

We now exploit the theory for Evolutionary Variational Inequalities (EVI)
developed in \cite{MurSav20GFEV} more precisely, 
we show that the theory developed there can be applied in the finite-dimensional ODE
setting and that these ODEs are equivalent to sublevel EVI where the convexity
parameter $\Lambda$ only depends on the sublevel, but not on $N$. Then we are
able to perform the limit $N\to \infty$ by using the established convergences 
$p^N(t)=S^N_t(\bbP_N p_0) \to p(t)=S_t(p_0)$ and $\calE(p^N(t))\to
\calE(p(t))$. This final limit is restricted to the case where we have a
uniform global sublevel stretching rate 
\[
 \Lambda^\glo(E):= \liminf_{N\to \infty} \Lambda^\glo_N(E) \quad
\text{with } \Lambda^\glo_N(E)=  \lim_{\Delta\to \infty}
\Lambda^\loc_N(\Delta,E),  
\]
see Definition \ref{de:GlobSublevSR}. Subsequently, we will assume $
\Lambda^\glo(E) >-\infty$ and emphasize that we have shown this only for
the case $k(p)= \kappa p$, where $\calD^k$ is a multiple of $\Bh$.

The following result is a simple adaption of three results in
\cite{MurSav20GFEV}, namely Eqn.\ (2.30), Proposition 3.11, and Lemma 3.13.
We provide the full (and simple) proof for the reader's convenience. We recall the definition of the right upper Dini derivative (see
\cite[App.\,A]{MurSav20GFEV}: 
\[
\frac{\rmd^+}{\rmd t} \,f(t) := \limsup_{h\searrow 0} \frac1h \big( f(t{+}h) -
f(t)\big). 
\]

\begin{proposition}[Sublevel EVI for the ODE in $\bfP_N$] 
\label{pr:SublevEVI:ODE}
Fix $N \in \bbN$, and consider the gradient-flow equation \eqref{eq:DiscrODE}
for the gradient system $\big( \bfP_N,\calE_N, \bbK_N\big)$. The global
stretching rates $\Lambda_N^\glo(E)$ are bounded from below by
$\wh\Lambda(N)>-\infty$, i.e.\ uniformly in
the energy level $E > \inf\calE_N$, and the solutions
$ p^N(t)=S^N_t(p_0^N)$ satisfy the sublevel EVI
\begin{align}
 \label{eq:SubEVI.ODE}
  \forall\, &E > \inf\calE_N\ \forall \, q^N, p^N_0\in \Sigma_N(E): 
 \\ \nonumber
 & \frac12 \frac{\rmd^+}{\rmd t}\calD^k_N(p^N(t), q)^2 + 
    \frac{\Lambda^\glo_N(E)}2\calD^k_N(p^N(t), q)^2 
     \leq \calE(q) - \calE(p^N(t)) \ \text{ a.e.\ in }{]0,\infty[}.
\end{align} 
\end{proposition}
\begin{proof} For notational simplicity we drop the super and subscripts $N$ in
  this proof, except at sets and spaces. The finiteness of $\Lambda^\glo_N$
  follows easily from the fact that $\bfP_N$ is a compact subset, hence the
  energy $\calE_N$ and the eigenvalues of $\bbH_N$ have finite lower
  bounds. Of course, these bounds need not be independent of $N$. 

\emph{Step 1:}
We first prove the sublevel version of \cite[Eqn.\,(2.30)]{MurSav20GFEV}. From
$p(t)=S_t(p_0), q\in \Sigma_N(E)$ and the definition of $\Lambda^\glo(E)$
we know that there exists a geodesic $s \mapsto \gamma(s) \in \bfP_N$
with $\gamma(0)=p(t)$ and $\gamma(1)=q$ such that 
\[
\calE(\gamma(s)) \leq (1{-}s) \calE(p(t)) + s \calE(q) -
\frac{\Lambda^\glo(E)}2\, s(1{-}s) \, \calD^k\bigl(p(t), q \bigr)^2.
\]
Subtracting $\calE(p(t))$ on both sides and dividing by $s>0$ we can pass to
the limit $s\searrow 0$ and find 
\[
\big\langle \rmD\calE(p(t)), \gamma'(0) \big\rangle \leq \calE(q) -
\calE(p(t)) - \frac{\Lambda^\glo(E)}2 \,\calD^k( p(t), q)^2. 
\]

\emph{Step 2:} Next we show that $s \mapsto \frac1s\frac{\rmd^+}{\rmd t}
\calD^k( p(t), \gamma(s))^2$ is nonincreasing, see
\cite[Lem.\,3.13]{MurSav20GFEV}. Because of $\calD^k \big( p(t),\gamma(s)
\big) = s \calD^k \big( p(t), q \big) $ and the continuity of $t \mapsto
p(t)$ we have  
\begin{align*}
\frac1s \frac{\rmd^+}{\rmd t} \, \calD^k \big( p(t), \gamma(s)\big)^2
& = \frac{2\calD^k(p(t),\gamma(s)\big)}{s} \:  
 \frac{\rmd^+}{\rmd t} \, \calD^k\big( p(t), \gamma(s)\big)
\\
& =  2\calD^k(p(t),q \big) \:  
 \frac{\rmd^+}{\rmd t} \, \calD^k\big( p(t), \gamma(s)\big).
\end{align*}
Thus it suffices to show monotonicity of $s \mapsto  \frac{\rmd^+}{\rmd t} \,
\calD^k\big( p(t), \gamma(s)\big)$. For $s_1< s_2$ we have 
\begin{align*}
&\calD^k \big(p(t{+}h),\gamma(s_2)\big) - \calD^k \big(p(t), \gamma(s_2) \big)
\\
&\leq \Big(\calD^k\big(p(t{+}h),\gamma(s_1)\big)+ 
 \calD^k\big(\gamma(s_1),\gamma(s_2)\big) \Big) -
 \Big(\calD^k\big(p(t), \gamma(s_1)) +\calD^k\big(\gamma(s_1),\gamma(s_2)\big) \Big) 
\\
&= \calD^k\big(p(t{+}h),\gamma(s_1)\big) - \calD^k\big(p(t), \gamma(s_1) \big) ,
\end{align*}
where we use the triangle inequality and the geodesic property of $\gamma$ with
$p(t)=\gamma(0)$. Dividing by $h>0$ and taking the limit $h \searrow 0$ gives the
desired monotonicity. 

\emph{Step 3:} Using Step 2 for $s=1$ (where $\gamma(1)=q$) and $s\searrow 0$
we obtain  
\[
\frac12\frac{\rmd^+}{\rmd t} \, \calD^k\big( p(t), q \big)^2 
\leq \lim_{s \searrow 0}
  \frac1{2s}\frac{\rmd^+}{\rmd t} \, \calD^k\big( p(t), \gamma(s) \big)^2 .
\]
In the last expression we only evaluate $\calD^k(p,g)$ for $p(\tau)$ and
$\gamma(s)$ near $p(t)$. Hence, the differentiabilities of $p$, $\gamma$, and
$\calD^k$ (locally near $p(t))$) provides the explicit expression 
\[
\lim_{s \searrow 0}
  \frac1{2s} \frac{\rmd^+}{\rmd t} \, \calD^k\big( p(t), \gamma(s) \big)^2  =
  - \big\langle \bbG(p(t)) \dot p(t), \gamma'(0) \big\rangle. 
\]
To see this, use that $\bfP_N$ is an $(N{-}1)$-dimensional affine
space and differentiate the approximation $\frac12 \big\langle
\bbG(p(t)) \,(p(\tau{+}h){-}\gamma(s)) , p(\tau{+}h) {-}\gamma(s) \big\rangle$ in $h$ and $s$.

\emph{Step 4: Conclusion.} Combining the results of Step 3, the gradient-flow
equation $\bbG \dot p=- \rmD\calE(p(t))$, and the estimate of Step 1 yields
\begin{align*}
\frac12\frac{\rmd^+}{\rmd t} \, \calD^k\big( p(t), q \big)^2 & \leq
-\big\langle \bbG(p(t)) \dot p(t), \gamma'(0) \big\rangle =  \big\langle \rmD
\calE(p(t)), \gamma'(0) \big\rangle 
\\
&\leq \calE(q) -
\calE(p(t)) - \frac{\Lambda^\glo(E)}2 \,\calD^k( p(t), q)^2. 
\end{align*}
This proves \eqref{eq:SubEVI.ODE}. 
\end{proof}

To pass to the limit $N\to \infty$, we make use of a derivative-free version of
the EVI provided in \cite[Thm.\,3.3]{MurSav20GFEV}. We will again use the
function $\mathsf M_\lambda(t) := \int_0^t \ee^{-\lambda s} \dd s$.  It can be
checked that the equivalence proof between the three different version is
compatible with our sublevel formulation, i.e.\ all constructions are contained
in a given sublevel. The following result heavily relies on the assumption that
the uniform global sublevel stretching rate $\Lambda^\glo(E)$ is larger than
$-\infty$. This we can only guarantee for the Bhattacharya case $k(p)=\kappa p$
under the additional condition \eqref{eq:W.doubling}.

\begin{theorem}[Sublevel EVI for the PDE]
\label{th:SublevelPDE} 
Under the assumptions \eqref{eq:k(p)bounds}, \eqref{eq:k.W.conds}, and the further assumption that $\Lambda^\glo(E)>-\infty $ for all $E>\inf \calE$, the solutions
$p(t)=S_t(p_0)$ with $p_0\in \dom(\calE)$ constructed in Theorem
\ref{th:SGLocContract} are sublevel EVI solutions for the metric gradient
system $ \big( \bfP, \calE, \calD^k\big)$ in the following sense:
\begin{align}
 \label{eq:SubEVI.PDE}
  \forall\, &E > \inf\calE\ \forall \, q, p_0\in \Sigma(E) \ \forall \, s\geq
  0\ \forall \, t>s: 
 \\ \nonumber
 & \frac{\ee^{\Lambda^\glo(E)(t-s)}}2 \calD^k(p(t), q)^2 - \frac12\calD^k(p(s), q)^2
    \leq \mathsf{M}_{\Lambda^\glo(E)}(t{-}s) \Big(\calE(q) - \calE(p(t))
  \Big) . 
\end{align} 
\end{theorem}
\begin{proof}
We first observe that \cite[Thm.\,3.3]{MurSav20GFEV} allows us to rewrite
\eqref{eq:SubEVI.ODE} in exactly the same derivative-free form, namely for all
$q^N\in \Sigma_N(E)$ we have
\begin{align*}
 & \frac{\ee^{\Lambda^\glo(E)(t-s)}}2 \calD_{LN}^k(p^N(t), q^N)^2 -
 \frac12\calD_{LN}^k(p^N(s), q^N)^2 
    \leq \mathsf{M}_{\Lambda^\glo(E)}(t{-}s) \Big(\calE_N(q^N) - \calE_N(p^N(t))
  \Big) ,
\end{align*}
where we already used that the the solutions $p^N$ are also solutions in
$\bfP_{LN}$ for all $L\in \N$ if $\calE$ is replaced by $\calE_N$. 
Moreover, we already used that the global sublevel stretching rate is uniform,
i.e.\ independent of $LN$. 

Now we let $M \geq N$ and choose $L$ such that $LN=M!$. Then, the limit
$M \to \infty$ allows us to replace $\calD^k_{LN}$ by the limit $\calD^k$ (cf.\
Proposition \ref{pr:def.calD.k}):
\begin{align*}
 & \frac{\ee^{\Lambda^\glo(E)(t-s)}}2 \calD^k(p^N(t), q^N)^2
  - \frac12\calD^k(p^N(s), q)^2  \leq 
 \mathsf{M}_{\Lambda^\glo(E)}(t{-}s) \Big(\calE_N(q^N) - \calE_N(p^N(t)) \Big) ,
\end{align*}

Finally, we can pass to the limit $N\to \infty$ by using $p^N(t)\to p(t)$ in
$\big(\bfP,\Bh \big)$ together with the energy convergence $\calE_N(p^N(t))\to
\calE(p(t))$ established in Theorem \ref{th:CurvesMaxSlope}. Moreover, we can
choose $q^N=\bbP_N q$ and exploit Lemma \ref{le:WellPrepIC} to find 
$\calE_N(q^N)\to \calE(q)$. 
\end{proof}

Of course, we can apply the equivalence in \cite[Thm.\,3.3]{MurSav20GFEV} once
again and obtain the more classical differential form of the EVI, namely 
\begin{align*}
% \label{eq:SubEVI.PDE.diff}
  \forall\, &E > \inf\calE\ \forall \, q, p_0\in \Sigma(E): 
 \\ \nonumber
 & \frac12 \frac{\rmd^+}{\rmd t}\calD^k(S_t(p_0), q)^2 + 
    \frac{\Lambda^\glo(E)}2\calD^k(S_t(p_0), q)^2 
     \leq \calE(q) - \calE(S_t(p_0)) \ \text{ a.e.\ in }{]0,\infty[}.
\end{align*}

As a consequence, we can obtain a full result for linear Eulerian viscosity
$\nu(\pl_x u) \pl_x \dot u$ with $\nu(p)=\hat\nu_0/p$ or $k(p)=\kappa p$ with
$\kappa =1/\hat \nu_0$. This case was treated in \cite{MiOrSe14ANVM} using the
minimizing movement scheme and relying totally on the 
(non-geodesic) Hellinger distance $\He|_\bfP$. As a result the EVI$_\lambda$ derived
there (see \cite[p.\,1336]{MiOrSe14ANVM}) contains an extra term $C_M
B_\xi(p(t),q)$ on the right-hand side. Our result below is based on the true
geodesic distance $\calD^k = \frac2{\sqrt{\kappa}} \Bh$  and thus provides 
a true EVI without extra terms. Moreover, for this case  Corollary
\ref{co:EnergyBhGeod} provides an explicit lower bound for $\Lambda^\glo(E)$,
namely $\frac4\kappa \,\bfL^\glo(E)$, see \eqref{eq:GlobStreRate}. For
notational simplicity the following result is formulated for the case
$k(p)=4p$, i.e.\ $\calD^k= \Bh$, the general case follows easily by rescaling
time.  

\begin{corollary}[EVI for $\calD^k=\Bh$] 
\label{co:Bhatt.EVI} Under the assumptions $k(p)=4p$, \eqref{eq:k.W.conds}, and \eqref{eq:W.doubling} we have 
$\Lambda^\glo(E)\geq  \bfL^\glo(E)>0$ (as defined in
\eqref{eq:GlobStreRate}). The solutions $p(t)=S_t(p_0)$ with $p_0\in
\dom(\calE)$ satisfy the sublevel EVI 
\begin{align}
 \label{eq:SubEVI.Bhatt}
  \forall\, &E > \inf\calE\ \forall \, q, p_0\in \Sigma(E): 
 \\ \nonumber
 & \frac12 \frac{\rmd^+}{\rmd t}\Bh(p(t), q)^2 + 
    \frac{\bfL^\glo(E)}2\,\Bh(p(t), q)^2 
     \leq \calE(q) - \calE(p(t)) \ \text{ a.e.\ in }{]0,\infty[}.
\end{align} 
\end{corollary}

\appendix 
\section{An example for \texorpdfstring{$\calD_{2M}|_{\bfP_2} <
    \calD^k_2$}{D(2M)|P2 less than P2}}
\label{se:calD.Counterexample}

Here we show that for general $k$ we cannot expect that $\calD^k_{NM}$ restricted
to $\bfP_N$ equals $\calD^k_N$. To provide an explicit example we let $N=2$
and consider $M\gg 1$. The function $k$ reads
\[
k(p)=4 p \ \text{ for } p\in[0,2]\qquad \text{and }\quad k(p)= \frac{p}{\eps} \ 
\text{ for } p>2,
\]  
with $\eps > 0$ chosen appropriately later. Moreover, we take start and end
points $p_0,p_1\in \bfP_2$: 
\[
    p_0(x) = 2\,\bbone_{{[1/2, 1[}}(x) \quad \text{and} \quad 
    p_1(x) = 2\,\bbone_{{[0, 1/2[}}(x).
\]
As $p_0p_1\equiv 0$ we have $\He(p_0,p_1)= \sqrt{2} \lneqq
\pi/2=\Bh(p_0,p_1)$, see Section \ref{se:Def.Bhatt}. 

For $M\gg 1$ we can define  a curve in $\bfP_{2M}$ that is close to the
Hellinger geodesics of the form  
\[
\wt p(s,x) = \alpha_M(s) \bbone_{{[0, 1/M[}}(x) + 2s^2 \bbone_{{[1/M,1/2[}}(x) +
\beta_M(s) \bbone_{{[1/2,1[}}(x) \quad \text{with } \alpha_M(s),\beta_M(s)\geq 0.
\]
Setting $s_M=M^{-1/2}$, we define $\alpha_M$ and $\beta_M$ as piecewise
polynomials via 
\begin{equation*}
  \alpha_M(s) = 2 s^2 M \text{ for } s \in [0, s_M], \quad
  \beta_M(s) = \beta_M(s_M) \frac{(1-s)^2}{(1-s_M)^2} \text{ for } s \in\, ]s_M, 1],
\end{equation*} 
defining $\alpha_M$ and $\beta_M$ on $]s_M, 1]$ and $[0, s_M]$ respectively via
the constraint
\begin{equation*}
  \frac{\alpha_M(s)}{M} + 2s^2 \biggl( \frac{1}{2} - \frac{1}{M} \biggr)
  + \frac{\beta_M(s)}{2} = 1,
\end{equation*}
which ensures $\wt{p}(s) \in \bfP_{2M}$ for all $s \in [0, 1]$. 

From $\alpha_M(0)=0$ we find $\wt p(0)=p_0$ and from $\beta_M(1)=0$ we find
$\wt p(1)=p_1$ as desired. Thus, $\calD^k_{2M}(p_0,p_1)^2 $ can be estimated
from above by 
\[
J^2:=\int_0^1 \hspace{-0.4em} \int_\Omega \frac{\bigl(\pl_s \wt p(s,x)\bigr)^2 }{ k\bigl( \wt
  p(s,x)\bigr) } \dd x \dd s.
\]
This integral can be calculated explicitly, but for proving the claim it
suffices to show that it is close to $\He(p_0,p_1)^2 $ and then exploit
$ \He(p_0,p_1)^2 \lneqq \Bh(p_0,p_1)^2$. 

Note that $\wt p(s,x)$ and in particular the local dissipation 
$\bigl(\partial_s\wt{p}(s, x)\bigr)^2/k(\wt{p}(s, x))$ are close to the
Hellinger geodesic $p^\He(s,x)=2s^2 \bbone_{{[0,1/2[}}(x) + 2(1{-}s)^2
\bbone_{{[1/2,1[}}(x) $ and its constant local dissipation 2,  
except for $x\in {[0,1/M[}$. However, we have $\alpha_M(s) >2 $ for $s\in
{]s_M,1[}$, because of $\alpha_M(s_M)=2=\alpha_M(1)=2$ and $\alpha''_M(s)\leq 4-2M<0$ for
$s\in {]s_M,1[}$. Exploiting this one find a constant $C>0$ such that for all
$M\geq 2$ and $\eps>0$ we have
$J^2\leq \He(p_0,p_1)^2 + C\big( M^{-1/2} + \eps M\big)$. Thus, choosing
$\eps=M^{-3/2}$ and making $M$ sufficiently large, we obtain
$\calD^k_{2M}(p_0,p_1) \leq J\lneqq \Bh(p_0,p_1) = \calD^k_2(p_0,p_1)$. 
   
\paragraph*{Acknowledgments.} AM was partially supported by the Deutsche
Forschungsgemeinschaft (DFG, German Research Foundation) under Germany's
Excellence Strategy \emph{The Berlin Mathematics Research Center MATH+}
(EXC-2046/1, project ID: 390685689). BS was supported by the MAC-MIGS Centre
for Doctoral Training under EPSRC grant EP/S023291/1.  AM is grateful for the
hospitality of Heriot-Watt University during a visit in 2024 where this work
was initiated. BS would like to thank John Ball for introducing him to this
topic, and for the helpful conversations and guidance.

\small
%\footnotesize

\bibliographystyle{alphaAMBS_doi}
\bibliography{visco_BS_AM2}

\end{document}